%% file: ex_article.tex
\documentclass[final]{siamart1116}

\input{ex_shared}

\ifpdf
\hypersetup{
	pdftitle={Low-complexity method for hybrid MPC with local guarantees},
	pdfauthor={Damian Frick, Angelos Georghiou, Juan L. Jerez, Alexander Domahidi, Manfred Morari},
	pdfcreator={pdfLaTeX},
	pdfsubject={},
	pdfkeywords={Operator splitting, Hybrid model predictive control, Non-smooth and discontinuous problems, Iterative schemes},
	colorlinks=false,
	breaklinks=true,
}
\fi

\begin{document}

\maketitle

% REQUIRED
% 244/250 words
\begin{abstract}
  Model predictive control problems for constrained hybrid systems are usually cast as mixed-integer optimization problems (MIP). However, commercial MIP solvers are designed to run on desktop computing platforms and are not suited for embedded applications which are typically restricted by limited computational power and memory. To alleviate these restrictions, we develop a novel low-complexity, iterative method for a class of non-convex, non-smooth optimization problems. This class of problems encompasses hybrid model predictive control problems where the dynamics are piece-wise affine (PWA). We give conditions such that the proposed algorithm has fixed points and show that, under practical assumptions, our method is guaranteed to converge locally to local minima. This is in contrast to other low-complexity methods in the literature, such as the non-convex alternating directions method of multipliers (ADMM), for which no such guarantees are known for this class of problems.
	By interpreting the PWA dynamics as a union of polyhedra we can exploit the problem structure and develop an algorithm based on operator splitting procedures. Our algorithm departs from the traditional MIP formulation, and leads to a simple, embeddable method that only requires matrix-vector multiplications and small-scale projections onto polyhedra.
	We illustrate the efficacy of the method on two numerical examples, achieving good closed-loop performance with computational times several orders of magnitude smaller compared to state-of-the-art MIP solvers. Moreover, it is competitive with ADMM in terms of suboptimality and computation time, but additionally provides local optimality and local convergence guarantees.
\end{abstract}

% REQUIRED
%\begin{keywords}
%  Operator splitting; Hybrid model predictive control; Non-smooth and discontinuous problems; Iterative schemes.
%\end{keywords}

% REQUIRED
%\begin{AMS}
%  93C30, 90C30, 49J52, 47H09
%\end{AMS}
% 93C30 Systems governed by functional relations other than differential equations (such as hybrid and switching systems)
% 90C30 Nonlinear programming
% 49J52 Nonsmooth analysis
% 47H09 Contraction-type mappings, nonexpansive mappings, A-proper mappings, etc

\section{Introduction}

Many practical applications for control fall in the domain of hybrid systems, including applications such as active suspension control or energy management in automotives \cite{dicairano2007, arce2009} and applications in power electronics \cite{geyer2008}. These applications require fast sampling times in the sub-second range, and the control algorithms need to be implemented on industrial, resource-constrained platforms, such as microcontrollers or re-configurable hardware. Hybrid systems are characterized by complex interactions between discrete and continuous behaviors that make them extremely challenging to control. In the last decades, model predictive control (MPC) \cite{morari1999} has received widespread attention both in research and industry. It provides a systematic approach for controlling constrained hybrid systems, promising high control performance with minimal tuning effort.

For hybrid systems, mature modeling tools such as the mixed logical dynamical framework, see \cite{bemporad1999}, are available. It defines rules for posing hybrid MPC as mixed-integer optimization problems (MIP). These MIPs can then be tackled with powerful commercial solvers such as CPLEX~\cite{cplex}.
These solvers have focused on the solution of generic, large-scale, mixed-integer linear/quadratic optimization problems on powerful computing platforms, and thus have several drawbacks for embedded applications:
\begin{enumerate*}[label=(\roman*)]
	\item their code size is in the order of tens of megabytes;
	\item the algorithms require substantial working memory;
	\item they depend on numerical libraries that cannot be ported to most embedded platforms.
\end{enumerate*}
This has restricted the applicability of hybrid MPC to systems that can run on desktop computing platforms, with sampling times in the order of minutes or even hours. 

Recently there have been efforts to enable hybrid MPC applications with faster dynamics to run on platforms with more limited computational power.
Explicit MPC \cite{alessio2009} is the method of choice for very small problems but quickly becomes intractable for increasing problem dimension.
Therefore, methods that further exploit
\begin{enumerate*}[label=(\roman*)]
	\item the sparsity-inducing multistage structure of MPC problems, and
	\item the fact that MPC problems are parametric and solved in a receding horizon fashion,
\end{enumerate*}
have been developed.
Methods based on reformulation and relaxation of MIPs arising from optimal control problems have been proposed in \cite{frick2015, axehill2008, jung2013, jung2013b}, and have led to reduced computation times compared to general-purpose solvers.
However, these methods are not fast enough for many practical applications.

Methods based on operator splitting allow to trade off computational complexity with optimality, i.e., they typically produce ``good'' solutions in a fraction of the time required even by tailored MIP solvers to find global optima.
Furthermore, they allow the problem to be split into several small subproblems, making these methods especially suited for embedded applications and implementations on computing systems with limited resources. The alternating directions method of multipliers (ADMM) \cite{boyd2011} is an example of such operator splitting methods. In \cite{takapoui2017}, ADMM was used on various hybrid MPC examples achieving substantial gains in computational performance compared to general-purpose MIP solvers.
Convergence guarantees are available for some special cases of non-convex ADMM \cite{li2015,magnusson2016}, however, neither optimality or convergence guarantees are available for the class of problems considered in this work.
Other low-complexity methods based on operator splitting suited for non-convex MPC are \cite{hours2016,houska2016}. They are attractive because they offer computational performance similar to ADMM, but unlike ADMM they still retain some guarantees on optimality and convergence. Nonetheless, they do not address the class of problems considered here.

\subsection{Contribution}
In this paper, we develop an algorithm based on operator splitting to address hybrid MPC problems, that provides both the desired theoretical guarantees and computational speed. Our method exploits the multistage structure of the optimization problem, allowing to decompose the complex piece-wise affine (PWA) equality constraints into decoupled non-convex polyhedral constraints. By leveraging their polyhedral nature, we are able to derive guarantees on optimality and convergence.
The main contributions of this paper are the following:
\begin{itemize}[wide,labelindent=0px]
	\item We develop an iterative, numerical method to find local minima of a class of non-convex, non-smooth optimization problems that arise in MPC problems for systems with PWA dynamics. Our method is of low complexity and only requires matrix-vector multiplications and small-scale Euclidean projections onto convex polyhedra.
	\item We provide
	\begin{enumerate*}[label=(\roman*)]
		\item conditions for the existence of fixed points of the proposed algorithm;
		\item show that the method converges locally to such fixed points; and
		\item prove that fixed points correspond to local minima.
	\end{enumerate*}
	To the author's knowledge, for the class of problems considered, this is the first method of such simplicity that allows to give local optimality and convergence guarantees.
	\item The proposed algorithm is implemented using automatic code generation, directly targeting the embedded applications. This code generation tool is made publicly available via github \cite{autolim}. We examine the method's performance on two numerical example and demonstrate computational speed-ups of several orders of magnitude, compared to conventional MIP solvers, with only a minor loss of optimality. We further show, that the method is competitive with other local methods such as ADMM.
\end{itemize}

\subsection{Outline}

In \refsec{sec:ProblemStatement}, we state the problem. In \refsec{sec:KKTOperator}, we construct an operator whose fixed points correspond to local minima. In \refsec{sec:method}, we propose an algorithm based on this operator and discuss local optimality and local convergence of the method. In \refsec{sec:numerical}, we present numerical results. Proofs of auxiliary results are given in the appendix.

\subsection{Notation \& elementary definitions}
For vectors $x,y \in \reals{n}$, $\|x\| := \sqrt{x^\transp x}$ is the 2-norm and $\langle x, y \rangle := x^\transp y$ the scalar product. For a closed set $\set{C}$, $\euclproj{\set{C}}(x) := \argmin_{z \in \set{C}} \|x-z\|$ is the Euclidean projection of $x$ onto $\set{C}$. The closed $\epsilon$-ball around $x$ is $\ball_{\epsilon}(x) := \{z \sep{} \|x-z\| \leq \epsilon\}$, we use the shorthand $\ball_{\epsilon} := \ball_{\epsilon}(0)$. The \emph{relative interior} is $\relint\set{C}$. By $\set{C}+\{x\}$ we denote the Minkowski sum of $\set{C}$ and the singleton $x$. The \emph{characteristic function} of set $\set{C}$ is $\charfunc{\set{C}}(x) := \{ 0 \text{ if } x \in \set{C}\,; +\infty \text{ otherwise}\}$. An \emph{operator} $T : \set{X} \rightrightarrows \set{Y}$ maps every point $x \in \set{X}$ to a set $T(x) \subseteq \set{Y}$. $T$ is called \emph{single-valued} if for all $x \in \set{X}$: $T(x)$ is a singleton. A \emph{zero} of $T$ is a point $x \in \set{X}$ such that $0 \in T(x)$, denoted by $x \in \zero T$, where $\zero T := \{x\sep{}0\in T(x)\}$. A \emph{fixed point} of $T$ is a point $x \in \set{X}$ such that $x \in T(x)$, denoted by $x \in \fix T$, where $\fix T := \{x\sep{}x\in T(x)\}$. 
Finally, when we say that a set has ``measure zero'' we more precisely mean that it has Lebesgue measure zero.

%%%%%%%%%%%%%%%%%     SECTION 2     %%%%%%%%%%%%%%%%%%%%%%%%%%
\section{Problem statement} \label{sec:ProblemStatement}

In this work, we consider optimization problems that are \emph{parametric} in $\theta \in \reals{p}$ and given as follows:
\begin{subequations} \label{prob:PrimalProblem}
	\begin{equation}
		\opt{p}(\theta) := \hspace{-1em}\min_{\substack{z \in \set{E} \cap \set{Z}(\theta)}} \cost{z}\,,
	\end{equation}
with $H \in \reals{n \times n}$ positive definite, and $z := (z_1,\ldots,z_N) \in \reals{n}$, with $z_k \in \reals{n_k}$. The parametric nature implies that only parts of the problem data will change every time the problem needs to be solved. The non-convex set $\set{Z}(\theta) \subseteq \reals{n}$ has the form:
\begin{equation}
	\set{Z}(\theta) := \bigtimes_{k=1}^{\horizon} \Big(\bigcup_{i=1}^{\NPwa_k} \set{Z}_{k,i}(\theta)\Big)\,, \label{eqn:Z}
\end{equation}
where the sets $\set{Z}_{k,i}(\theta) \subseteq \reals{n_k}$ are closed convex polyhedra that may depend linearly on $\theta$, in the following way: $\set{Z}_{k,i}(\theta) := \{ z_k \in \reals{n_k} \sep{} G_{k,i}z_k= g_{k,i}(\theta)\,, F_{k,i}z_k \leq f_{k,i}(\theta) \}$, with matrices and vectors of appropriate dimensions, and $g_{k,i}(\theta)$, $f_{k,i}(\theta)$ affine functions of the parameter $\theta$.
For a given parameter $\theta$, the Euclidean projection onto $\set{Z}(\theta)$ can be evaluated very efficiently by decoupling it into simple convex projections, see \refalg{alg:projection} in \refsec{sec:method}. This important feature of $\set{Z}$ will be a key ingredient of the algorithm presented in \refsec{sec:method}.
For simplicity of notation we will omit the parameter dependence of $\set{Z}(\theta)$ in the remainder of this paper.
Finally, the set $\set{E}$ is an affine equality constrained set
\begin{equation}
	\set{E} := \{ z \in \reals{n} \sep{} Az =b\} = \{ Vv + \bar{v} \sep{} v \in \reals{n-m} \}\,,
\end{equation}
with $A \in \reals{m\times n}$ full row rank, and an appropriate matrix $V \in \reals{n\times n-m}$ and vector $\bar{v} \in \reals{n}$ \cite[p.~430, (15.14)]{numericalOptimization2006}, where $V$ has full column rank by construction. We furthermore assume that $\set{E} \cap \set{Z}$ is non-empty, hence the problem is feasible.
\end{subequations}

\begin{remark}
	The class of problems named \emph{Mathematical Programs with Complementarity Constraints} (MPCCs) can be brought into the form of \refprob{prob:PrimalProblem}, as long as their constraint functions are linear and their objective function is strictly convex quadratic. This can be done by constructing a union of sets that enumerates the different possibilities of the complementarity constraints. However, in general it is not possible to represent an MPCC in the form of \refprob{prob:PrimalProblem} with $N > 1$ and low-dimensional sets $\set{Z}_{k,i}$. This representation is therefore often, and in general, exponential in the number of complementarity constraints. In contrast, with the representation \refprob{prob:PrimalProblem} we explicitly consider this Cartesian structure and in this work develop a method that exploits it.
\end{remark}

\subsection{Consensus form}

\refprob{prob:PrimalProblem} can be re-written in the so-called consensus form by introducing copies $y$ of $z$ as follows
\begin{equation} \label{prob:ConsensusProblem}
	\opt{p} := \min_{z,y\,:\,z=y} \cost{z} + \charfunc{\set{E}}(z) + \charfunc{\set{Z}}(y)\,,
\end{equation}
which is a non-convex, non-smooth optimization problem. The constraint sets $\set{E}$ and $\set{Z}$ have been moved into the cost function using their characteristic functions $\charfunc{\set{E}}(z)$ and $\charfunc{\set{Z}}(y)$. This problem formulation ensures that the two sets $\set{E}$ and $\set{Z}$ can be decoupled. The consensus constraint $z=y$ encodes the coupling between $z$ and $y$.

\subsection{Hybrid Model Predictive Control} \label{sec:mpc}

We are particularly interested in MPC problems where the discrete-time dynamics are given by a PWA function of the states and inputs. Such parametric hybrid MPC problems can be written as:
\begin{subequations} \label{prob:MPCProblem}
	\begin{align}
		\min_{x_k,u_k}& \quad \sum_{k=1}^N\nolimits q_{k+1}(x_{k+1}) + r_{k}(u_k)\\[-0.4em]
		\suchthat &\quad x_{k+1} = A_{k,i} x_k + B_{k,i} u_k + c_{k,i} \quad \text{ if } (x_k,u_k) \in \set{C}_{k,i}\,,\: k=1,\ldots,N \,,\label{eqn:MPCPwa}\\
		&\quad x_1 = \theta\,,
	\end{align}
\end{subequations}
for $i=1,\ldots,m_k$ with $m_k$ representing the number of regions of the PWA dynamics. Vector $x_k \in \reals{n_x}$ denotes the state of the system at time $k$, $u_k \in \reals{n_u}$ denotes the inputs applied to the system between times $k$ and $k+1$, and $\theta \in \reals{n_x}$ is the (parametric) initial state of the system. The objective terms $q_{k+1}$ and $r_{k}$ are strictly convex, quadratic functions. The sets $\set{C}_{k,i}$ are non-empty closed convex polyhedra that define the partition of the PWA dynamics, and include state and input constraints.

Problems of the form \eqref{prob:MPCProblem} can be solved by introducing additional continuous and binary variables, and formulating an MIP using, e.g., the big-M reformulation \cite{bemporad1999} or more advanced formulations \cite{vielma2015}. The resulting problems can then be solved using off-the-shelf commercial MIP solvers. 
Instead of formulating a MIP, in this work we transform \refprob{prob:MPCProblem} into the form~(\ref{prob:PrimalProblem}). In this way only $\horizon n_x$ auxiliary continuous variables and \emph{no binary variables} need to be introduced. For each time instant $k$ we introduce a copy $w_k$ of $x_{k+1}$, as follows: $w_k := A_{k,i} x_k + B_{k,i} u_k + c_{k,i}$, if $(x_k,u_k) \in \set{C}_{k,i}$.
This allows us to define the closed convex polyhedral sets
\begin{align*}
	\set{Z}_{1,\hat{\imath}}(\theta) &:= \big\{ (u_1,w_1) \sep{\big} w_1 = A_{1,\hat{\imath}} \theta + B_{1,\hat{\imath}} u_1 + c_{1,\hat{\imath}} \text{ and }(\theta,u_1) \in \set{C}_{1,\hat{\imath}}\big\}\,,\\
	\set{Z}_{k,i} &:= \big\{ (x_k,u_k,w_k) \sep{\big} w_k = A_{k,i} x_k + B_{k,i} u_k + c_{k,i} \text{ and } (x_k,u_k) \in \set{C}_{k,i}\big\}\,,
\end{align*}
for $\hat{\imath} = 1,\ldots,m_1$ and for $i=1,\ldots,m_k$, $k=2,\ldots,N$.
The non-convex constraint set $\set{Z}$ now has the form of \eqref{eqn:Z} and is \emph{decoupled} in the horizon $N$.
For an equivalent formulation, we impose $x_{k+1}=w_k$ by defining a set of \emph{coupling} equality constraints
\begin{equation*}
	\set{E} := \bigtimes_{k=1}^N \Big( \reals{n_u} \times \big\{(w_k,x_{k+1}) \in \reals{2n_x} \sep{\big} x_{k+1}=w_k\big\} \Big)\,.
\end{equation*}

\begin{remark} \label{rem:MPCConsCost}
To ensure the strict convexity of \refprob{prob:PrimalProblem} equivalent to \refprob{prob:MPCProblem}, we impose the objective $\alpha_k q_{k+1}(x_{k+1}) + (1-\alpha_k)q_{k+1}(w_{k}) + r_k(u_k)$ for each $k$ and any $\alpha_k \in (0,1)$. With the coupling constraint $w_k=x_{k+1}$ this ensures that the cost remains unchanged. Without loss of generality, we have used $\alpha_k = 0.5$ for all $k$.
\end{remark}

\begin{remark}[Properties of the non-convex set $\set{Z}$]
	In this work we deal with the general case of $\set{Z}$ being polyhedral non-convex, and we do not make explicit use of structure of $\set{Z}$ beyond the fact that it can be represented as the Cartesian product over unions of closed convex polyhedra.
	In \refsec{sec:method}, we give local convergence guarantees under assumptions that additionally restrict the set $\set{Z}$. However, even with these restrictions, $\set{Z}$ can still be polyhedral non-convex, which in particular means that it can have holes and gaps.
	Nevertheless, depending on the application, additional structure of $\set{Z}$ may be desirable for the presented method to be useful from a practical standpoint. For example, when the set $\set{Z}$ consists of many disjoint sets, then there exist many local minima, some of which may be almost trivial to find. Indeed, we have found that for MPC problems the method performs well if the piecewise affine dynamics are continuous, which often leads to fewer local minima.
	%In the numerical results we consider an example with discontinuous dynamics, \refexa{exa:bemporad}, as well as an example with continuous dynamics, \refexa{exa:racecars}.
\end{remark}

%%%%%%%%%%%%%%%%%     SECTION 3     %%%%%%%%%%%%%%%%%%%%%%%%%%
\section{KKT conditions for \refprob{prob:ConsensusProblem}} \label{sec:KKTOperator}
Our aim is to develop an algorithm to find local minima of \refprob{prob:ConsensusProblem}.
For convex optimization, Karush-Kuhn-Tucker (KKT) conditions give rise to necessary and sufficient conditions for global optimality, under mild assumptions. In the non-convex, non-smooth case, however, they are in general neither necessary nor sufficient, even for local optimality.
We use a notion of KKT points for instances of \refprob{prob:ConsensusProblem} that provide necessary conditions for local optimality. Computing such KKT points is in general intractable. Therefore, in this section, we propose two restrictions (inner approximations) that are computationally attractive.
Based on these KKT conditions, we construct an operator $K_{\xi}$ whose properties and structure allow us to derive a computationally efficient method and give local optimality and convergence guarantees in \refsec{sec:method}.

\subsection{Generalized, regular \& proximal KKT points}  \label{sec:KKTPoints}
We denote KKT points by~$\kkt{}$, local optima by~$\loc{}$ and global optima by~$\opt{}$.
A point $(\kkt{z},\kkt{y},\kkt{\lambda})$ is called a generalized KKT point \cite[Theorem~3.34, p.~127]{ruszczynski2006} of \refprob{prob:ConsensusProblem} if it satisfies
%\begin{subequations} \label{def:origKKTpoints}
%\begin{align}
%	0 &\in \subd \left(\cost{z} + \charfunc{\set{E}}(z) + \charfunc{\set{Z}}(y)\right) (\kkt{z},\kkt{y})\nonumber\\
%	&\quad- \{\nabla\langle \kkt{\lambda}, z-y\rangle(\kkt{z},\kkt{y})\}\,,\label{eqn:KKTstationarity}\\
%	0 &= \kkt{z}-\kkt{y}\,,\label{eqn:KKTprimalfeasibility}
%\end{align}
%\end{subequations}
%where \eqref{eqn:KKTstationarity} is called the \emph{stationarity} condition and \eqref{eqn:KKTprimalfeasibility} is the \emph{primal feasibility} condition.
%We can simplify \eqref{def:origKKTpoints}, by using $\kkt{z} = \kkt{y}$, and the following two facts about subdifferentials:
%\begin{enumerate*}[label=(\roman*)]
%	\item the subdifferential $\subd f(x)$ of the sum $f:=g+f_0$ of a finite function $g$ and a smooth function $f_0$ is simply $\subd f(x) = \subd g(x) + \{\nabla f_0(x)\}$, see \cite[Exercise~8.8, p.~304]{rockafellar1998};
%	\item the sum of two separable functions $f(x) + g(y)$ has a decoupled subdifferential $\subd f(x) \times \subd f(y)$, see \cite[Proposition~10.5, p.~426]{rockafellar1998}.
%\end{enumerate*}
%This implies that \eqref{def:origKKTpoints} is equivalent to
\begin{subequations} \label{def:KKTpoints}
\begin{align}
	0 &\in \{H \kkt{z} + h\} + \normalcone{\set{E}}(\kkt{z}) - \{\kkt{\lambda}\}\,,\\
	0 &\in \normalcone{\set{Z}}(\kkt{y}) + \{\kkt{\lambda}\}\,,\label{eqn:KKTpoints2}\\
	0 &= \kkt{z}-\kkt{y}\,,
\end{align}
\end{subequations}
where the regular normal cone $\rnormalcone{\set{C}}(z)$ and the (limiting) normal cone $\normalcone{\set{C}}(z)$ of a set~$\set{C}$ are defined as in~\cite[Proposition~6.5, p.~200]{rockafellar1998}:
\begin{equation*}
	\rnormalcone{\set{C}}(z) := \big\{ x \in \reals{n} \sep{\big} \limsup_{u \lattentive{\set{C}} z} \frac{\langle x, u-z \rangle}{\|u-z\|}\big\}\,, \text{ and } \normalcone{\set{C}}(z) := \limsup_{x\lattentive{\set{C}} z} \rnormalcone{\set{C}}(x) = \subd \charfunc{\set{C}}(z)\,.
\end{equation*}
In particular, the KKT conditions \eqref{def:KKTpoints} imply that $\kkt{z}=\kkt{y} \in \set{E} \cap \set{Z}$. Since any KKT point $(\kkt{z},\kkt{y},\kkt{\lambda})$ satisfies primal feasibility, $\kkt{z}=\kkt{y}$, we will use the shorthand notation $(\kkt{z},\kkt{\lambda})$ for KKT points in the remainder of this paper.
The KKT conditions \eqref{def:KKTpoints} give rise to necessary conditions for local optimality as detailed below.
\begin{lemma}[Necessary condition for local optimality] \label{lem:KKTExistence}
	For every local minimum $\loc{z}$ of \refprob{prob:PrimalProblem} there exist Lagrange multipliers $\loc{\lambda} \in \reals{n}$ such that $(\loc{z},\loc{\lambda})$ is a generalized KKT point of \refprob{prob:ConsensusProblem}.
\end{lemma}
These KKT conditions are difficult to exploit due to the dependence on $\normalcone{\set{Z}}(\cdot)$, which cannot be evaluated easily \cite{henrion2008}. However, we can define two restrictions that admit practical characterizations, resulting in what we call \emph{regular} and \emph{$\xi$-proximal} KKT points. These special cases of KKT points can be evaluated more efficiently and have important regularity properties. Furthermore, in \refsec{sec:locopt}, we will show that they provide sufficient (but no longer necessary) conditions for local optimality.
We call a point ($\kkt{z},\kkt{\lambda}$) a \emph{regular} KKT point of \refprob{prob:ConsensusProblem}, if it satisfies
\begin{subequations} \label{def:RegularKKTPoints}
\begin{align}
	0 &\in \{H \kkt{z} + h\} + \rnormalcone{\set{E}}(\kkt{z}) - \{\kkt{\lambda}\}\,,\text{ and}\\
	0 &\in \rnormalcone{\set{Z}}(\kkt{z}) + \{\kkt{\lambda}\}\,.
\end{align}
\end{subequations}
Note that we have replaced the normal cone $\normalcone{\set{Z}}(\kkt{z})$ from \eqref{eqn:KKTpoints2} with the regular normal cone $\rnormalcone{\set{Z}}(\kkt{z})$. Since $\rnormalcone{\set{Z}}(\kkt{z}) \subseteq \normalcone{\set{Z}}(\kkt{z})$, see \cite[Proposition~6.5, p.~200]{rockafellar1998}, the set of regular KKT points is an inner approximation, a restriction, of the set of generalized KKT points.
Given $\xi>0$, we call ($\kkt{z},\kkt{\lambda}$) a $\xi$-\emph{proximal} KKT point of \refprob{prob:ConsensusProblem}, if it satisfies
\begin{subequations} \label{def:ProximalKKTPoints}
	\begin{align}
		0 &\in \{H \kkt{z} + h\} + \rnormalcone{\set{E}}(\kkt{z}) - \{\kkt{\lambda}\}\,, \text{ and}\label{eqn:ProximalKKTPoint1}\\
		\kkt{z} &\in \euclproj{\set{Z}}(\kkt{z}-\tfrac{1}{\xi}\kkt{\lambda})\,.\label{eqn:ProximalKKTPoint2}
	\end{align}
\end{subequations}
	Compared to the definition of regular KKT points we have replaced the regular normal cone $\normalcone{\set{Z}}(\kkt{z})$ with a condition on the projection, and we have introduced a positive scaling $\xi$ using $\kkt{z} \in \euclproj{\set{Z}}(\kkt{z}-\tfrac{1}{\xi}\kkt{\lambda}) \Rightarrow -\tfrac{1}{\xi}\kkt{\lambda} \in \rnormalcone{\set{Z}}(\kkt{z})$, from \cite[Example~6.16, p.~212]{rockafellar1998}. This directly implies that every $\xi$-proximal KKT point $(\kkt{z},\kkt{\lambda})$ is also a regular KKT point and therefore \eqref{def:ProximalKKTPoints} is a restriction of \eqref{def:RegularKKTPoints}.
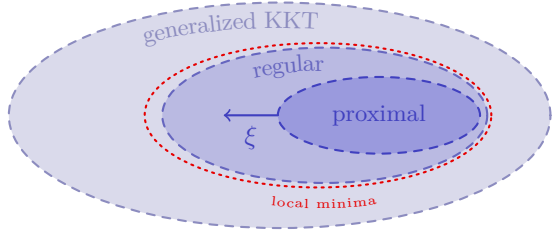
\begin{figure}
	\centering
	\begin{tikzpicture}[scale=0.6,cap=round]
	%\draw[style=help lines,step=0.5em] (-1.4,-1.4) grid (1.4,1.4);
	\fill[fill=cplexcolorlight] (0,0) ellipse (6 and 2.5);
	\draw[color2,dotted,thick,postaction={decorate,decoration={raise=-1.8ex,text along path,text align={center,left indent=5.5cm},text={|\tiny\color{color2}|local minima}}}] (0.85,0) ellipse (3.84 and 1.6);
	\draw[cplexcolor,dashed,thick,postaction={decorate,decoration={raise=-2.2ex,text along path,reverse path,text align={center,left indent=7cm},text={|\small\color{cplexcolor}|generalized KKT}}}] (0,0) ellipse (6 and 2.5);
	\filldraw[mosekcolor,fill=mosekcolorlight,dashed,thick,postaction={decorate,decoration={raise=-2ex,text along path,reverse path,text align={center,left indent=4cm},text={|\small\color{mosekcolor}|regular}}}] (1,0) ellipse (3.6 and 1.5);
	\filldraw[gurobicolor,fill=gurobicolorlight,dashed,thick] (2.2,0) ellipse (2.24 and 0.85) node {\small{proximal}};
	\draw[thick,->,gurobicolor] ($(2.2,0)-(2.24,0)$) -- ($(2.2,0)-(2.24,0)-(1.2,0)$) node[midway,below] {$\xi$};
	\end{tikzpicture}
	\caption{The sets illustrate the relationship between different KKT points and represent points $z$ for which there exist multipliers $\lambda$ such that $(z,\lambda)$ is a generalized, regular or proximal KKT point ({\color{blue}dashed}), as well as the set of local minima ({\color{color2}dotted}).} \label{fig:KKTrelation}
\end{figure}

Evaluating whether a pair $(z,\lambda)$ is a regular KKT point has an exponential worst-case complexity, because representing $\rnormalcone{\set{Z}}(z)$ may be combinatorial and can quickly become intractable. In contrast, evaluating whether a pair $(z,\lambda)$ is a $\xi$-proximal KKT point amounts to a projection that can be performed in polynomial-time. Furthermore, for the primal variables $z$, this inner approximation \eqref{def:ProximalKKTPoints} of \eqref{def:RegularKKTPoints} can be made tight, as shown in the following proposition. \reffig{fig:KKTrelation} illustrates the conceptual relationship between the different KKT points.
\begin{proposition} \label{prop:proxNormalsExistence}
	Consider \refprob{prob:ConsensusProblem}.
	\begin{enumerate}[label=\roman*),ref=\ref{prop:proxNormalsExistence}(\roman*)]
		\item\label{prop:proxNormalsExistencei} For any regular KKT point $(\kkt{z},\kkt{\lambda})$, there exists a $\bar{\xi}>0$ such that for all $\xi \geq \bar{\xi}$, $(\kkt{z},\kkt{\lambda})$ is a $\xi$-proximal KKT point.
		\item\label{prop:proxNormalsExistenceii} Given any $\xi>0$ and $\xi$-proximal KKT point $(\kkt{z},\kkt{\lambda})$, then $(\kkt{z},\kkt{\lambda})$ is also a regular KKT point.
	\end{enumerate}
\end{proposition}
\refprob{prob:ConsensusProblem} has only finitely many local minima due to positive definiteness of $H$ and the polyhedrality of $\set{E} \cap \set{Z}$. Therefore, given a $\xi>0$ large enough, for every regular KKT point $(\kkt{z},\kkt{\lambda})$ there exists a $\kkt{\mu}$ such that $(\kkt{z},\kkt{\mu})$ is a $\xi$-proximal KKT point.

\subsection{Operator for proximal KKT points}
\begin{subequations}
The fact that set-membership evaluation of \eqref{def:ProximalKKTPoints} is efficient, motivates the construction of an operator $K_{\xi}$ that has the property that
\begin{enumerate*}[label=(\roman*)]
	\item it is cheap to evaluate and 
	\item has zeros corresponding to $\xi$-proximal KKT points of \refprob{prob:ConsensusProblem}.
\end{enumerate*}
We use these two properties in \refsec{sec:method} to construct an algorithm based on this operator for finding local minima of \refprob{prob:PrimalProblem}.
We define a set-valued operator $K_{\xi} : \reals{n} \rightrightarrows \reals{n}$:
\begin{equation} \label{def:K}
	K_{\xi}(s) := \{M_{\xi}s + c_{\xi}\} - \euclproj{\set{Z}}(s)\,,
\end{equation}
with $\xi>0$ the \emph{proximal scaling}, and
\begin{align}\label{eqn:Mrhoxi}
	M_{\xi} &:= \xi\left( \xi R - I\right)^{-1}R\,,\quad c_{\xi} := \left(\xi R - I\right)^{-1}\left(R (h+H\bar{v}) -\bar{v}\right)\,, \\
	R &:= V(V^\transp H V)^{-1}V^\transp\,,
\end{align}
\end{subequations}
where $V$ and $\bar{v}$ as in~\refsec{sec:ProblemStatement}. We will show that finding an $s$ such that $0 \in K_\xi(s)$ is equivalent to finding a $\xi$-proximal KKT point $(\kkt{z},\kkt{\lambda})$.
For $K_\xi$ to be defined properly, $\xi R-I$ needs to be invertible.
Since $R \succeq 0$, this holds for all
%\begin{subequations}
%\begin{equation} \label{eqn:xiCond}\
%	\xi > \frac{1}{\lambda^+_{\rm min}(R)}\,,
%\end{equation}
%where $\lambda^+_{\rm min}(R)$ is the smallest non-zero eigenvalue of $R$. Finally, we note that~\eqref{eqn:xiCond} implies $M_{\xi} \succeq 0$ and
%\begin{equation} \label{eqn:Meig}
%	\lambda_{\rm min}^+(M_\xi) > 1\,.
%\end{equation}
%\end{subequations}
\begin{subequations}
\begin{align}
	&\xi > \left(\lambda^+_{\rm min}(R)\right)^{-1}\,.\label{eqn:xiCond}\\
	\text{This implies} &\: M_{\xi} \succeq 0 \text{ and } \lambda_{\rm min}^+(M_\xi) > 1\,,\label{eqn:Meig}
\end{align}
\end{subequations}
where $\lambda^+_{\rm min}(R)$ is the smallest non-zero eigenvalue of $R$.
For \refprob{prob:MPCProblem}, it can be shown that~$\xi > \lambda_{\rm max}(H)$ is necessary and sufficient for satisfying both~\eqref{eqn:xiCond} and~\eqref{eqn:Meig}.
%To achieve the theoretical guarantees in \refsec{sec:method}, we will need to impose further lower bounds on $\xi$.
Using \eqref{def:K} we can show that a point $\kkt{s}\in\zero K_{\xi}$, called a zero of $K_\xi$, corresponds to a $\xi$-proximal KKT point $(\kkt{z},\kkt{\lambda})$ with $\kkt{z}$, $\kkt{\lambda}$ as follows:
\begin{equation}\label{eqn:zerosK}
	\kkt{z} := M_{\xi}\kkt{s} + c_{\xi} \text{ and } \kkt{\lambda} := \xi(\kkt{z}-\kkt{s})\,,
\end{equation}
where the precise relationship is given in \reflem{lem:opKKT}.
\begin{lemma}[Zeros of operator $K_{\xi}$] \label{lem:opKKT}
	Given any $\xi>0$. A point $\kkt{s} \in \reals{n}$ satisfies $\kkt{s} \in \zero K_{\xi}$ if and only if the pair $(\kkt{z},\kkt{\lambda})$ constructed through \eqref{eqn:zerosK} is a $\xi$-proximal KKT point of \refprob{prob:ConsensusProblem}.
\end{lemma}
Multiple points $\kkt{s} \in \zero K_{\xi}$ can result in the same $\kkt{z}$ (but different $\kkt{\lambda}$), since $M_{\xi}$ is only invertible when $\set{E}=\reals{n}$. This means in particular that there can be multiple zeros of $K_\xi$ corresponding to the same local optimum.

%%%%%%%%%%%%%%%%%     SECTION 4     %%%%%%%%%%%%%%%%%%%%%%%%%%
\section{Solution method and theoretical guarantees} \label{sec:method}
In this section we present an algorithm constructed to find $\xi$-proximal KKT points of \refprob{prob:ConsensusProblem} by finding zeros of the operator $K_{\xi}$. The basis of this algorithm is a fixed-point iteration, where every iteration is computationally inexpensive. In Theorem~4.2, we show that if the proposed algorithm converges, then it converges to a local minimum of \refprob{prob:PrimalProblem}. Furthermore, in Theorem~4.6 we show that, under practical assumptions, the presented algorithm converges locally to such minima, almost always. By `almost always' we mean that, for any local minimum that has corresponding fixed points, there exist at most a measure zero subset of fixed points from which the method may not converge.

\subsection{Fixed-point algorithm}

\begin{subequations} \label{def:fixedpointOp}
We define an operator $T_{\xi} : \reals{n} \rightrightarrows \reals{n}$, that has the zeros of $K_{\xi}$ as its fixed points:
\begin{equation}
	T_{\xi}(s) := s - WK_{\xi}(s) = s-W\big(M_{\xi}s+c_{\xi}-\euclproj{\set{Z}}(s)\big)\,,\label{eqn:T2}
\end{equation}
with $W \in \reals{n \times n}$ invertible to ensure $\fix T_{\xi} = \zero K_{\xi}$. $W$ is defined as
\begin{equation} \label{eqn:W}
	W := Q\begin{bsmallmatrix}\frac{1}{2}\Lambda^{-1} & 0 \\ 0 & -I\end{bsmallmatrix}Q^\transp\,,
\end{equation}
where $Q \in \reals{n\times n}$ is orthonormal and $\Lambda \in \reals{n-m\times n-m}$ diagonal, positive definite, such that $M_{\xi} = Q \diag(\Lambda,0)Q^\transp$ is the eigendecomposition of $M_{\xi}$. In addition, the structure of $W$ will be used to prove local nonexpansiveness of the operator $T_\xi$, see \reflem{lem:niceop} in \refsec{sec:convergence}.
The local nonexpansiveness of $T_\xi$ allows us to use the \emph{Krasnoselskij iteration} \cite[Theorem~3.2, p.~65]{berinde2007} to find fixed points of $T_\xi$. Given an initial iterate $s_0 \in \reals{n}$, the Krasnoselskij iteration is defined as follows:
\begin{equation} \label{eqn:krasonselskij}
	s_{j+1} = (1-\gamma)s_j + \gamma T_{\xi}(s_j)\,,
\end{equation}
\end{subequations}
with step size $\gamma \in (0,1)$. From \eqref{eqn:krasonselskij} we obtain \refalg{alg:krasnoselskij}.

\begin{algorithm}[H]
	\caption{Low-complexity method for \refprob{prob:ConsensusProblem}} \label{alg:krasnoselskij}
  	\begin{algorithmic}[1]
   		\Require{$s_0 \in \reals{n}$, $\gamma \in (0,1)$, $\epsilon_{\rm tol} > 0$}
   		%\If{$\opt{z} := \bar{v}-R(h+H\bar{v}) \in \set{Z}$} \Return $\opt{z}$ \Comment{Check trivial solution} \EndIf \label{alg:stepGlobal}
   		\State \algorithmicif\ $\opt{z} := \bar{v}-R(h+H\bar{v}) \in \set{Z}$ \algorithmicthen\ \Return $\opt{z}$ \Comment{check trivial solution}\label{alg:stepGlobal}
   		\Repeat{ \textbf{for} $j=0,1,\ldots$}
    			\State $z_{j+1} = M_{\xi}s_j + c_{\xi}$\label{alg:stepAff}
    			\State $y_{j+1} \in \euclproj{\set{Z}}(s_j)$ \Comment{projection, see Alg.~\ref{alg:projection}}\label{alg:stepProjection}
    			\State $s_{j+1} = s_j - \gamma W(z_{j+1}-y_{j+1})$ \Comment{Krasnoselskij iteration}\label{alg:stepKrasnoselskij}
    		\Until{$\|z_{j+1}-y_{j+1}\| \leq \epsilon_{\rm tol}$} \algorithmicthen\ \Return $y_{j+1}$ \Comment{termination criterion}
  	\end{algorithmic}
\end{algorithm}
Each iteration of \refalg{alg:krasnoselskij} is simple and geared towards an efficient embedded implementation utilizing automatic code-generation. It involves only matrix-vector multiplications and a small number of projections. The projection onto the non-convex set $\set{Z}$ in step \ref{alg:stepProjection} is the most expensive step. It can be evaluated in polynomial time using \refalg{alg:projection}, where at most $\sum_{k=1}^N m_k$ small-scale projections onto polyhedra are needed (\refalg{alg:projection}, step~\ref{alg:smallProjection}). State-of-the-art embedded solvers \cite{domahidi2012} or even explicit solutions \cite{herceg2013b} can be used for these projections.
For problems of the form \eqref{prob:MPCProblem}, the matrices $M_{\xi}$ and $W$ are block-banded with bandwidth $\max\{2n_x,n_u\}$, therefore steps \ref{alg:stepAff} and \ref{alg:stepKrasnoselskij} of \refalg{alg:krasnoselskij} can be evaluated very efficiently and have a computational complexity of $\bigo (N(n_x+n_u)^2 )$. This enables the automatic generation of efficient, embeddable code that is tailored to particular (parametric) problems.

\begin{algorithm}[H]
	\caption{$\euclproj{\set{Z}}(s)$, projection of $s$ onto $\set{Z}$} \label{alg:projection}
  	\begin{algorithmic}[1]
   		\Require{$s = (s_1,\ldots,s_N) \in \reals{n}$, $\set{Z}$ of form \eqref{eqn:Z}}
   		\For{$k=1,\ldots,\horizon$}
   			\State \algorithmicfor\ $i=1,\ldots,\NPwa_k$ \algorithmicdo\ $z_{k,i} = \euclproj{\set{Z}_{k,i}}(s_k)$ \Comment{small-scale projections}\label{alg:smallProjection}
   				%\State $z_{k,i} = \euclproj{\set{Z}_{k,i}}(s_k)$ \Comment{small-scale projection}\label{alg:smallProjection}
   			%\EndFor
   			\State $z_k \in \argmin_{i=1,\ldots,\NPwa_k} \|s_k-z_{k,i}\|$ \Comment{select closest}
   		\EndFor
   		\State \Return $z = (z_1,\ldots,z_\horizon)$
  	\end{algorithmic}
\end{algorithm}

%It should be noted that while the simplicity of \refalg{alg:krasnoselskij} is reminiscent of ADMM, the algorithms have important differences that make it possible to give local optimality and convergence guarantees in our case.

\subsection{Local optimality}\label{sec:locopt}
First, we will show that when \refalg{alg:krasnoselskij} converges, it converges to a local minimum, and equivalently that the fixed points of the operator $T_{\xi}$ correspond to local minima.
Any set $\set{Z}$ of form \eqref{eqn:Z} can be transformed into a finite union of closed, convex polyhedra $\set{Z}^i \subseteq \reals{n}$ with $\set{Z} = \bigcup_{i=1}^{I} \set{Z}^i$, where $I:=\prod_{k=1}^N m_k$.
This equivalent representation helps to simplify our theoretical considerations. The exact construction of the sets $Z^i$ does not affect the presented results.
Furthermore, we define the \emph{set of active components} of $\set{Z}$ at a point $z \in \set{Z}$ as $\activePoly{\set{Z}}(z) := \{i \in \{1,\ldots,I\} \sep{} z \in \set{Z}^i\}$.
To establish the subsequent theoretical results we will use the polar $\prnormalcone{\set{Z}}$ of the regular normal cone, defined in \cite[Equation~6(14), p.~215]{rockafellar1998} as $\prnormalcone{\set{Z}}(z) =\{v \sep{} \langle v,w \rangle \leq 0\,,\:\forall w \in \rnormalcone{\set{Z}}(z)\}$.
The following lemma establishes two important properties of the set $\set{Z}$ and its normal cone. These properties will be used in \refthm{thm:localOptimality} for proving local optimality. At any point $z\in\set{Z}$, \reflem{lem:convNCone} relates the set $\set{Z}$ to a local convexification $\prnormalcone{\set{Z}}(z)+\{z\}$, as illustrated in \reffig{fig:localConv}.
\begin{figure}
	\centering
	\begin{subfigure}[b]{0.46\columnwidth}
	\centering
	\begin{tikzpicture}[scale=0.45,cap=round]
	%\draw[style=help lines,step=0.5em] (-1.4,-1.4) grid (1.4,1.4);
	\fill[fill=color1light] (2,0) -- (3,1) -- (3,-0.5) -- cycle;
	\draw[very thick,color1] (2,0) -- (3,1) -- (3,-0.5) -- cycle;
	\fill[fill=color1light] (-1.5,-0.5) -- (-2.5,-2.5) -- (-3.5,-1.5) -- (-2.5,-0.5) -- cycle;
	\draw[very thick,color1] (-1.5,-0.5) -- (-2.5,-2.5) -- (-3.5,-1.5) -- (-2.5,-0.5) -- cycle;
	\fill[fill=color3verylight] (0,0) -- (-1.7910,-3.5766) arc (243.4:315:4);
		\node[below,color3dark] at (0.4,-2.5) {$\rnormalcone{\set{Z}_1}(z_1)$};
	\draw[color3light,very thick] (0,0) -- (-1.3433,-2.6825);
		\draw[color3light,very thick,dashed] (-1.3881,-2.7719) -- (-1.7910,-3.5766);
		\draw[color3light,very thick] (0,0) -- (2.1213,-2.1213);
		\draw[color3light,very thick,dashed] (2.1920,-2.1920) -- (2.8284,-2.8284);
	\fill[fill=verylightdarkgreen] (0,0) -- (2.475,2.475) arc (45:153.4:3.5);
		\node[above,darkgreen] at (-0.4,1.7) {$\prnormalcone{\set{Z}_1}(z_1)+\{z_1\}$};
	\draw[very thick,color1] (-3,1.5) -- (0,0);
	\draw[very thick,color1] (0,0) -- (2,2) node[below]{$\set{Z}_1$};
	\draw[very thick,color1] (-1.1,-2.8) -- (2.9,-0.8);
	\draw[thick,dotted,black] (0,0) circle (1.5);
		\draw[<->,thick,black] (0,0) -- (0,1.5) node[midway,left,black]{$\epsilon$};	
	\node at (0,0) [circle,fill=color2,inner sep=1] {};
		\node at (0,0) [right,color=color2]{$z_1$};	
	\end{tikzpicture}
	\end{subfigure}
	\begin{subfigure}[b]{0.46\columnwidth}
	\centering
	\begin{tikzpicture}[scale=0.45,cap=round]
	%\draw[style=help lines,step=0.5em] (-1.4,-1.4) grid (1.4,1.4);
	\draw[lightdarkgreen,very thick] (-1,-1) -- (2.2,2.2);
		\node[above,darkgreen] (rncname) at (0.2,3) {$\prnormalcone{\set{Z}_2}(z_2)+\{z_2\}$};
		\draw[lightdarkgreen,very thick,dashed] (-1.5,-1.5) -- (-1.1,-1.1);
		\draw[lightdarkgreen,very thick,dashed] (2.3,2.3) -- (2.6,2.6) node[midway] (rnc) {};
	\draw[darkgreen,thick] (rncname.south) edge[bend right,out=-90,in=-200,->] ($(rnc)+(-0.1,-0.05)$);
	\draw[very thick,color1] (-3.5,1.0) -- (-0.5,-0.5);
	\draw[very thick,color1] (-0.5,-0.5) -- (2,2) node at (2.15,1.5) {$\set{Z}_2$};
	\draw[very thick,color1] (-1.1,-2.8) -- (2.9,-0.8);
	\fill[fill=color1light] (2,0) -- (3,1) -- (3,-0.5) -- cycle;
	\draw[very thick,color1] (2,0) -- (3,1) -- (3,-0.5) -- cycle;
	\fill[fill=color1light] (-1.5,-0.5) -- (-2.5,-2.5) -- (-3.5,-1.5) -- (-2.5,-0.5) -- cycle;
	\draw[very thick,color1] (-1.5,-0.5) -- (-2.5,-2.5) -- (-3.5,-1.5) -- (-2.5,-0.5) -- cycle;
	\draw[color3light,very thick] (-1,2) -- (2.5,-1.5);
		\node[below,color3dark] at (2.8,-1.8) {$\rnormalcone{\set{Z}_2}(z_2)$}; 
		\draw[color3light,very thick,dashed] (-1.1,2.1) -- (-1.5,2.5);
		\draw[color3light,very thick,dashed] (2.6,-1.6) -- (3,-2);
	\draw[thick,dotted,black] (0.5,0.5) circle (1.3);
		\draw[<->,thick,black] (0.5,0.5) -- (0.5,1.8) node[midway,left,black]{$\epsilon$};
	\node at (0.5,0.5) [circle,fill=color2,inner sep=1] {};
		\node at (0.45,0.5) [right,color=color2]{$z_2$};	
	\node[] at (0,-3.5) {};
	\end{tikzpicture}
	\end{subfigure}
	\caption{Two cases of local convexifications ${\color{darkgreen}\prnormalcone{\set{Z}}(z)+\{z\}}$ of non-convex sets ${\color{color1}\set{Z}_1}$ and ${\color{color1}\set{Z}_2}$ at points ${\color{color2}z_1}$ (left) and ${\color{color2}z_2}$ (right), as described in \reflem{lem:convNCone}.} \label{fig:localConv}
\end{figure}
\begin{lemma}[Local convexification] \label{lem:convNCone}
For any $z\in\set{Z}$, the set $\set{Z}$ and the closed convex set $\prnormalcone{\set{Z}}(z)+\{z\}$ are related in the following ways:
	\begin{enumerate}[label=\roman*),ref=\ref{lem:convNCone}(\roman*)]
		\item\label{lem:convNConei} There exists an $\epsilon >0$ such that $\set{Z} \cap \ball_\epsilon(z) \subseteq \prnormalcone{\set{Z}}(z) + \{z\}$.
		\item\label{lem:convNConeii} Their regular normal cones at $z$ coincide, i.e., $\rnormalcone{\set{Z}}(z) = \rnormalcone{\prnormalcone{\set{Z}}(z)+\{z\}}(z)$.
	\end{enumerate}
\end{lemma}

We now use \reflem{lem:convNCone} to prove the first key result.
\begin{theorem}[Local optimality]\label{thm:localOptimality}
	If \refalg{alg:krasnoselskij} converges to a point $(\loc{y},\loc{z},\loc{s})$, then $\loc{y}=\loc{z}$ and $\loc{z}$ is a local minimum of \refprob{prob:PrimalProblem}. Moreover, $(\loc{z},\xi(\loc{z}-\loc{s}))$ is a regular KKT point of~\refprob{prob:ConsensusProblem}.
\end{theorem}
\begin{proof}
When \refalg{alg:krasnoselskij} converges to a point $(\loc{y},\loc{z},\loc{s})$, we have by definition that $\loc{z} = M_{\xi}\loc{s} + c_{\xi}$, $\loc{y} \in \euclproj{\set{Z}}(\loc{s})$ and $\loc{s} = \loc{s} - \gamma W(\loc{z}-\loc{y})$.
Together with the invertibility of $W$, it follows that $\loc{z} = \loc{y}$. Moreover, from the definition of $K_\xi$ it follows that $0 \in K_{\xi}(\loc{s})$, i.e., $\loc{s} \in \zero K_\xi$. \reflem{lem:opKKT} thus implies that $(\loc{z},\loc{\lambda})$, with $\loc{\lambda} := \xi(\loc{z},\loc{s})$ is a $\xi$-proximal KKT point of \refprob{prob:ConsensusProblem}, which implies that $(\loc{z},\loc{\lambda})$ is a regular KKT point of \refprob{prob:ConsensusProblem}, see \refsec{sec:KKTPoints}.
Given this, we now consider the auxiliary problem: 
%	\begin{subequations} \label{prob:locallyconvex}
%	\begin{align} 
%		\min_{z,y} \:&\cost{z}\\[-0.5em]
%		\suchthat & (z,y) \in \set{E} \times \left(\prnormalcone{\set{Z}}(\loc{z})+\{\loc{z}\}\right)\,,\\
%		& z=y\,.
%	\end{align}
%	\end{subequations}
%	\begin{equation} \label{prob:locallyconvex}
%		\min_{\substack{z,y\,\suchthat\,z=y\,,\\(z,y) \in \set{E} \times \left(\prnormalcone{\set{Z}}(\loc{z})+\{\loc{z}\}\right)}} \hspace*{-2em} \cost{z}\,.
%	\end{equation}
	\begin{equation} \label{prob:locallyconvex}
		\min_{z,y\,:\, z=y} \cost{z} + \charfunc{\set{E} \times (\prnormalcone{\set{Z}}(\loc{z})+\{\loc{z}\})}(z,y)\,.
	\end{equation}
	\refprob{prob:locallyconvex} is strictly convex and its KKT conditions are given as $z=y$ (primal feasibility) and
	\begin{equation} \label{eqn:polarKKT}
		0 \in \{Hz+h\} + \rnormalcone{\set{E}}(z) - \{\lambda\} \text{ and } 0 \in \rnormalcone{\prnormalcone{\set{Z}}(\loc{z})+\{\loc{z}\}}(y) + \{\lambda\}\,.
	\end{equation}
	By \reflem{lem:convNConeii}, \eqref{eqn:polarKKT} is equivalent to the regular KKT conditions~\eqref{def:RegularKKTPoints}. Therefore, $(\loc{z},\loc{\lambda})$ also satisfies \eqref{eqn:polarKKT}. By strict convexity, implied by $H \succ 0$, $(\loc{z},\loc{\lambda})$ is primal-dual optimal for \refprob{prob:locallyconvex}. Furthermore, using \reflem{lem:convNConei}, we know that there exists an $\epsilon>0$ such that $\loc{z} \in \set{Z} \cap \ball_\epsilon(\loc{z}) \subseteq \prnormalcone{\set{Z}}(\loc{z})+\{\loc{z}\}$, therefore, the pair $y=z=\loc{z}$ is also optimal for the non-convex problem
%	\begin{subequations} \label{prob:localopt}
%	\begin{align}
%		\min_{z,y} \:&\cost{z}\\[-0.5em]
%		\suchthat & (z,y) \in \set{E} \times \left(\set{Z} \cap \ball_{\epsilon}(\loc{z})\right)\,,\\
%		& z=y\,.
%	\end{align}
%	\end{subequations}
	\begin{equation*} \label{prob:localopt}
		\min_{z,y\,:\, z=y} \cost{z} + \charfunc{\set{E} \times \left(\set{Z} \cap \ball_{\epsilon}(\loc{z})\right)}(z,y)\,,
	\end{equation*}
	which is a local instance of \refprob{prob:ConsensusProblem}, where the feasible region is restricted to $\ball_{\epsilon}(\loc{z})$. This implies that $\loc{z}$ is locally optimal for \refprob{prob:ConsensusProblem} and \refprob{prob:PrimalProblem}.
\end{proof}

\subsection{Existence}
We have seen that for any regular KKT point we can ensure the existence of a $\xi$-proximal KKT point by choosing $\xi$ large enough. Therefore, if there exists a regular KKT point, we can ensure the existence of fixed points of the operator $T_\xi$, and therefore of \refalg{alg:krasnoselskij}, as demonstrated in the following theorem.
\begin{theorem}[Existence of fixed points] \label{thm:existence}
	If there exists a regular KKT point $(\kkt{z},\kkt{\lambda})$ of \refprob{prob:ConsensusProblem}, then there exists a $\bar{\xi} > 0$ such that \refalg{alg:krasnoselskij}, and equivalently $T_{\xi}$, has at least one fixed point for any $\xi \geq \bar{\xi}$.
\end{theorem}
\begin{proof}
	Given a regular KKT point $(\kkt{z},\kkt{\lambda})$ of \refprob{prob:ConsensusProblem}. By \refprop{prop:proxNormalsExistencei}, there exists a $\bar{\xi}>0$ such that for any $\xi\geq\bar{\xi}$ we have that $(\kkt{z},\kkt{\lambda})$ is a $\xi$-proximal KKT point. By \reflem{lem:opKKT}, $\kkt{s}:=\kkt{z}-\tfrac{1}{\xi}\kkt{\lambda} \in \zero K_{\xi}$ and thus $\kkt{s}$ is a fixed point of $T_{\xi}$ and $(\kkt{z},\kkt{y},\kkt{s})$, with $\kkt{y} = \kkt{z} = M_\xi \kkt{s} + c_\xi \in \euclproj{\set{Z}}(\kkt{s})$ is a fixed point of \refalg{alg:krasnoselskij}.
\end{proof}
In \reffig{fig:KKTpoints} we illustrate four examples for the existence and nature of KKT points. Figures \ref{fig:KKTpointsA}--\ref{fig:KKTpointsC} are cases where the proximal scaling $\xi > 0$ can be chosen such that fixed points of \refalg{alg:krasnoselskij} exist and we may find local minima using the proposed method. In \reffig{fig:KKTpointsD} the feasible set is a singleton and none of its corresponding KKT points are regular. In this case \refalg{alg:krasnoselskij} would not have a fixed point for any $\xi >0$, even though a generalized KKT point exists. Figures \ref{fig:KKTpointsA} and \ref{fig:KKTpointsB} illustrate common cases. On the other hand Figures \ref{fig:KKTpointsC} and \ref{fig:KKTpointsD} depict more rare cases.

\begin{figure}
	\centering
	\begin{subfigure}[b]{0.23\textwidth}
	\centering
	\begin{tikzpicture}[scale=0.5,cap=round]
	%\clip ($(-3.5,-1.5)+(0,-11em)$) rectangle (2.5,4.8);
	\clip (-3.5,-1.5) rectangle (2.5,4.8);
	%\draw[style=help lines,step=0.5em] (-1.4,-1.4) grid (1.4,1.4);
	\coordinate (xs) at (-0.5,1);
	\foreach \i in {5,15,25,35,45,55}{
		\fill[fill=orange!\i] (xs) circle ({(65-\i)/25});};
	\node[green!60!black] at (-2,2) {$\set{E}=\reals{2}$};
	\draw[very thick] (-3,0) -- (0,0) node[midway,below]{$\set{Z}_1$};
	\draw[very thick,color1] (0,0) -- (2,2) node[below=0.1]{$\set{Z}_2$};
	\node at (-0.5,0) [circle,fill=color2,inner sep=1] {}; \node at (-0.5,0) [above,color=color2]{$z_1$};
	\draw[->,thick,color2] (-0.5,0) -- (-0.5,-1) node[right]{$\lambda_1$};
	\node at (0.25,0.25) [circle,fill=color2,inner sep=1] {}; \node at (0.25,0.25) [right,color=color2]{$z_2$};
	\draw[->,thick,color2] (0.25,0.25) -- (1,-0.5) node[right]{$\lambda_2$};
	
	\begin{scope}[xshift=-3.6cm,yshift=2cm] 
		\clip (0,0) rectangle (7,2.8);
		\begin{axis}[scale=0.9,hide axis] 
			\addplot[black,domain=-3:0, no markers,samples=20] {0.5*x^2+(0.5)*x}; 
			\addplot[black,dotted,domain=0:1, no markers,samples=20] {0.5*x^2+(0.5)*x}; 
			\addplot[color1,domain=0:2, no markers,samples=20] {x^2+(0.5-1)*x}; 
			\addplot[color1,dotted,domain=-1:0, no markers,samples=20] {x^2+(0.5-1)*x};
			\node[circle,fill=color2,inner sep=1.8] at (axis cs:-0.5,-0.125) {};
			\node[circle,fill=color2,inner sep=1.8] at (axis cs:0.25,-0.0625) {};
		\end{axis}
	\end{scope}
	\end{tikzpicture}
	\caption{	} \label{fig:KKTpointsA}
	%\caption{There are two regular KKT points, $({\color{color2}z_1},{\color{color2}\lambda_1})$ and $({\color{color2}z_2},{\color{color2}\lambda_2})$. Both correspond to local minima.} \label{fig:KKTpointsA}
	\end{subfigure}
	\hspace*{0.01\textwidth}
	\begin{subfigure}[b]{0.23\textwidth}
	\centering
	\begin{tikzpicture}[scale=0.5,cap=round]
	\clip (-3.5,-1.5) rectangle (2.5,4.8);
	%\draw[style=help lines,step=0.5em] (-1.4,-1.4) grid (1.4,1.4);
	\coordinate (xs) at (0,1);
	\foreach \i in {5,15,25,35,45,55}{
		\fill[fill=orange!\i] (xs) circle ({(65-\i)/25});};
	\node[green!60!black] at (-2,2) {$\set{E}=\reals{2}$};
	\draw[very thick] (-3,0) -- (0,0) node[midway,below]{$\set{Z}_1$};
	\draw[very thick,color1] (0,0) -- (2,2) node[below=0.1]{$\set{Z}_2$};
	\node at (0.5,0.5) [circle,fill=color2,inner sep=1] {}; \node at (0.5,0.5) [right,color=color2]{$z_1$};
	\draw[->,thick,color2] (0.5,0.5) -- (1,0) node[right]{$\lambda_1$};
	\node at (0,0) [circle,fill=color2,inner sep=1] {}; \node at (-0.1,0) [above,color=color2]{$z_2$};
	\draw[->,thick,color2] (0,0) -- (0,-1) node[right]{$\lambda_2$};
	
	\begin{scope}[xshift=-3.6cm,yshift=1.5cm] 
		\clip (0,0) rectangle (7,3);
		\begin{axis}[scale=0.9,hide axis]
			\addplot[black,domain=-3:0, no markers,samples=20] {0.5*x^2+(0)*x}; 
			\addplot[color1,domain=0:2, no markers,samples=20] {x^2+(0-1)*x}; 
			\addplot[black,dotted,domain=0:1, no markers,samples=20] {0.5*x^2+(0)*x}; 
			\addplot[color1,dotted,domain=-1:0, no markers,samples=20] {x^2+(0-1)*x}; 
			\node[circle,fill=color2,inner sep=1.8] at (axis cs:0,0) {};
			\node[circle,fill=color2,inner sep=1.8] at (axis cs:0.5,-0.25) {};
		\end{axis}
	\end{scope}
	\end{tikzpicture}
	\caption{} \label{fig:KKTpointsB}
	%\caption{There is one regular KKT point $({\color{color2}z_1},{\color{color2}\lambda_1})$, corresponding to a local minimum and one non-regular KKT point $({\color{color2}z_2},{\color{color2}\lambda_2})$, corresponding to a critical point that is not locally optimal.} \label{fig:KKTpointsB}
	\end{subfigure}
	\hspace*{0.01\textwidth}
	\begin{subfigure}[b]{0.23\textwidth}
	\centering
	\begin{tikzpicture}[scale=0.5,cap=round]
	\clip (-3.5,-1.5) rectangle (2.5,4.8);
	%\draw[style=help lines,step=0.5em] (-1.4,-1.4) grid (1.4,1.4);
	\coordinate (xs) at (-0.5,0.5);
	\foreach \i in {5,15,25,35,45,55}{
		\fill[fill=orange!\i] (xs) circle ({(65-\i)/25});};
	\draw[very thick,green!60!black,dashed] (2.1,2.1) -- (-1,-1) node[below,near end]{$\set{E}$};
	\draw[very thick] (-3,0) -- (0,0) node[midway,below]{$\set{Z}_1$};
	\draw[very thick,color1] (0,0) -- (2,2) node[below=0.1]{$\set{Z}_2$};
	\node at (0,0) [circle,fill=color2,inner sep=1] {}; \node at (0,0) [right,color=color2]{$z$};
	\draw[->,thick,color2] (0,0) -- (0.5,-0.5) node[right]{$\lambda_1$};
	\draw[->,thick,color2] (0,0) -- (-0.5,0.5) node[above]{$\lambda_2$};
	
	\begin{scope}[xshift=-3.6cm,yshift=1.2cm]
		\clip (0,0) rectangle (7,3);
		\begin{axis}[scale=0.9,hide axis]
			\addplot[green!60!black,domain=-2:2, no markers,samples=20,dotted] {x^2}; 
			\addplot[black,domain=-3:0, no markers,samples=20,dotted] {0.5*x^2+(0.5)*x}; 
			\addplot[color1,domain=0:2, no markers,samples=20] {x^2}; 
			\node[circle,fill=color2,inner sep=1.8] at (axis cs:0,0) {};
		\end{axis}
	\end{scope}
	\end{tikzpicture}
	\caption{} \label{fig:KKTpointsC}
	%\caption{There is one non-regular KKT point $({\color{color2}z},{\color{color2}\lambda_1})$ corresponding to the unique global minimum. However $({\color{color2}z},{\color{color2}\lambda_2})$ is a regular KKT point that corresponds to the same global minimum.} \label{fig:KKTpointsC}
	\end{subfigure}	
	\hspace*{0.01\textwidth}
	\begin{subfigure}[b]{0.23\textwidth}
	\centering
	\begin{tikzpicture}[scale=0.5,cap=round]
	\clip (-3.5,-1.5) rectangle (2.5,4.8);
	%\draw[style=help lines,step=0.5em] (-1.4,-1.4) grid (1.4,1.4);
	\coordinate (xs) at (-0.5,0.5);
	\foreach \i in {5,15,25,35,45,55}{
		\fill[fill=orange!\i] (xs) circle ({(65-\i)/25});};
	\draw[very thick,green!60!black,dashed] (-3,3) -- (1,-1) node[below,near start]{$\set{E}$};
	\draw[very thick] (-3,0) -- (0,0) node[midway,below]{$\set{Z}_1$};
	\draw[very thick,color1] (0,0) -- (2,2) node[below=0.1]{$\set{Z}_2$};
	\node at (0,0) [circle,fill=color2,inner sep=1] {}; \node at (0,0) [right,color=color2]{$z$};
	\draw[->,thick,color2] (0,0) -- (0,-1) node[left]{$\lambda_1$};
	\draw[->,thick,color2] (0,0) -- (0.5,-0.5) node[right]{$\lambda_2$};
	
	\begin{scope}[xshift=-3.6cm,yshift=1.2cm]
		\clip (0,0) rectangle (7,3);
		\begin{axis}[scale=0.9,hide axis]
			\addplot[green!60!black,domain=-2.5:1.5, no markers,samples=20,dotted] {x^2+x}; 
			\addplot[black,domain=-3:0, no markers,samples=20,dotted] {0.5*x^2+(0.5)*x}; 
			\addplot[color1,domain=0:2, no markers,samples=20,dotted] {x^2};
			\node[circle,fill=color2,inner sep=1.8] at (axis cs:0,0) {};
		\end{axis}
	\end{scope}
	\end{tikzpicture}
	\caption{} \label{fig:KKTpointsD}
	%\caption{There are only two KKT points $({\color{color2}z},{\color{color2}\lambda_1})$ and $({\color{color2}z},{\color{color2}\lambda_2})$ corresponding to the unique global minimum, both are non-regular.} \label{fig:KKTpointsD}
	\end{subfigure}	
	\caption{KKT points of $\min_{z,y \in {\color{green!60!black}\set{E}}\times\left({\color{black}\set{Z}_1} \cup {\color{color1}\set{Z}_2}\right)\,:\,x=y} \tfrac{1}{2}z^\transp z + h^\transp z$. {\color{color3}Level curves} of the objective and the objective value evaluated on ${\color{black}\set{Z}_1}$, ${\color{color1}\set{Z}_2}$ and ${\color{green!60!black}\set{E}}$ are illustrated in the respective colors, solid where feasible, dotted where infeasible.
	(\subref{fig:KKTpointsA}) $({\color{color2}z_1},{\color{color2}\lambda_1})$ and $({\color{color2}z_2},{\color{color2}\lambda_2})$ are regular KKT points corresponding to local minima.
	(\subref{fig:KKTpointsB}) $({\color{color2}z_1},{\color{color2}\lambda_1})$, is a regular KKT point corresponding to a local minimum, $({\color{color2}z_2},{\color{color2}\lambda_2})$ a non-regular KKT point  corresponding to a (not locally optimal) critical point.
	(\subref{fig:KKTpointsC}) $({\color{color2}z},{\color{color2}\lambda_1})$, $({\color{color2}z},{\color{color2}\lambda_2})$ are a non-regular and a regular KKT point, corresponding to the same global minimum.
	(\subref{fig:KKTpointsD}) $({\color{color2}z},{\color{color2}\lambda_1})$, $({\color{color2}z},{\color{color2}\lambda_2})$ are the only KKT points corresponding to the unique global minimum, both are non-regular.}
	\label{fig:KKTpoints}
\end{figure}

\subsection{Local convergence to local minima} \label{sec:convergence}
In this section, we prove that the presented method converges locally to local minima.
We say that a fixed point $\loc{s}$ \emph{corresponds} to a local minimum $\loc{z}$, if the $\xi$-proximal KKT point $(\loc{z},\loc{\lambda})$ can be constructed from $\loc{s}$ through \eqref{eqn:zerosK}. Note that, as shown in \reflem{lem:opKKT} and \refthm{thm:localOptimality}, any fixed point of $T_\xi$ can be used to reconstruct a local minimum. We denote the set of fixed points corresponding to $\loc{z}$ by $\fixc{T_\xi}{\loc{z}} := \{ s \sep{} (\loc{z},\xi(\loc{z}-s)) \text{ satisfies \eqref{def:ProximalKKTPoints}}\}$. The set $\fixc{T_\xi}{\loc{z}}$ is convex for fixed $\loc{z}$, because $\rnormalcone{\set{E}}(\loc{z})$ is a convex set for any $\loc{z}$ and because the set of $\lambda$'s satisfying $\loc{z} \in \euclproj{\set{Z}}(\loc{z}-\tfrac{1}{\xi}\lambda)$ form the \emph{proximal normal cone} of $\set{Z}$ at $\loc{z}$ which is also convex, see \cite[p.~213]{rockafellar1998}.
We make the following assumptions:
\begin{enumerate}[label=\textbf{(A\arabic*)},ref=A\arabic*]
	\item\label{ass:existence} \refprob{prob:ConsensusProblem} has at least one regular KKT point.
	\item\label{ass:nondeg} Given any local minimum $\loc{z}$ of \refprob{prob:PrimalProblem}  such that $\fixc{T_{\bar{\xi}}}{\loc{z}} \neq \varnothing$ and $\loc{z} \not\in \fixc{T_{\bar{\xi}}}{\loc{z}}$ for some $\bar{\xi} > 0$. There exists a $\xi \geq \bar{\xi}$ and a fixed point $\loc{s}\in \fixc{T_\xi}{\loc{z}}$ with multipliers $\loc{\lambda} := \xi(\loc{z}-\loc{s})$ such that $-\loc{\lambda}\in\relint\rnormalcone{\set{Z}}(\loc{z})$.
%	\item\label{ass:nondeg} There exists a $\bar{\xi}>0$ such that
%		\begin{enumerate}[label=(\roman*)]
%			\item\label{ass:nondeg1} $\bar{\xi} > {\lambda_{\rm min}^{+}(R)}^{-1}$, i.e., \eqref{eqn:xiCond},
%			\item\label{ass:nondeg2} $\bar{\xi}$ satisfies \refprop{prop:xilb}, and
%			\item\label{ass:nondeg3} for any local optimum $\loc{z}$, and any $\xi \geq \bar{\xi}$, for which $\fixc{T_\xi}{\loc{z}} \neq \varnothing$, there exists a fixed point $\loc{s}\in \fixc{T_\xi}{\loc{z}}$ that satisfies $\loc{s}\in\relint\rnormalcone{\set{Z}}(\loc{z}) + \{\loc{z}\}$.
%		\end{enumerate}
	\item\label{ass:Z}  The structure of set $\set{Z}$ is such that, for any $z \in \set{Z}$, either $\rnormalcone{\set{Z}}(z) = \{0\}$, or there exists an $\epsilon > 0$ such that for all $w \in \rnormalcone{\set{Z}}(z)^\orth \cap \ball_\epsilon$, we have $z+w \in \set{Z}$.
\end{enumerate}
Assumption~\refass{ass:existence} ensures the existence of regular KKT points. It furthermore implies the existence of fixed points of $T_\xi$ for $\xi$ large enough, according to \refthm{thm:existence}. 
Assumption~\refass{ass:nondeg} can be understood as a \emph{non-degeneracy} condition on the multipliers.
Assumption~\refass{ass:Z} is an assumption on the geometry of $\set{Z}$ and ensures mild local regularity.
\begin{figure}
	\centering
	\begin{tikzpicture}[scale=0.5,cap=round]
	%\draw[style=help lines,step=0.5em] (-1.4,-1.4) grid (1.4,1.4);
	\fill[color1light!50] (0,0) -- (1,1) -- (-3,1) -- (-4,0) -- cycle;
	\draw[color1light,thick] (0,0) -- (1,1) -- (-3,1) -- (-4,0) -- cycle;
	\fill[color1light!50] (0,-1.4) -- (0,1.4) -- node[midway,above,blue] {$\set{Z}$} (-4,1.4) -- (-4,-1.4) -- cycle;
	\draw[color1light,thick] (0,-1.4) -- (0,1.4) -- (-4,1.4) -- (-4,-1.4) -- cycle;
	\fill[color1light!50] (-1,-1) -- (0,0) -- (-4,0) -- (-5,-1) -- cycle;
	\draw[color1light,thick] (-1,-1) -- (0,0) -- (-4,0) -- (-5,-1) -- cycle;
	\draw[color1light,thick] (0,0) -- (-4,0);
	\draw[color3light,very thick] (0,0) -- (2.5,0);
	\draw[color3light,very thick,dotted] (2.5,0) -- node[midway,above,color3dark] {$\rnormalcone{\set{Z}}(z)$} (3,0);
	\draw[->,thick,black,dashed] (0,0) -- (0.5,1.2) node[midway, right] {$w$};
	\node at (0,0) [circle,fill=color2,inner sep=1] {}; \node at (0.2,0) [color2, below] {$z$};
	\end{tikzpicture}
	\caption{A set ${\color{blue}\set{Z}}$ where \refass{ass:Z} is violated at point ${\color{color2}z}$. We see that $w$ is orthogonal to ${\color{color3dark}\rnormalcone{\set{Z}}(z)}$, but ${\color{color2}z}+\epsilon w \not\in {\color{blue}\set{Z}}$ for any $\epsilon > 0$.}
	\label{fig:ZCounterExample}
\end{figure}
These assumptions are substantially less restrictive than common assumptions for non-convex optimization methods such as smoothness, Clarke- or prox-regularity of the constraint set $\set{E} \cap \set{Z}$ at critical points.
In fact, we show in \refsec{sec:nondegeneracy} that \refass{ass:nondeg} holds for any given \refprob{prob:PrimalProblem} with (almost) arbitrary convex quadratic objective and in \refapp{app:assCheck} we present~\refalg{alg:assZCheck} which provides a necessary and sufficient condition for \refass{ass:Z}.
\refalg{alg:assZCheck} is a combinatorial algorithm that enumerates the active sets and components of $\set{Z}$. It may therefore perform badly for anything but small dimensions. However, when the sets $\set{Z}_{k,i}$ of $\set{Z}$ are of low dimension, which is often the case in structured problems such as hybrid MPC problems, the check can be performed separately for each $k$. This significantly reduces the computational burden. Moreover, this check can be performed \emph{once and offline} for all parameters $\theta$.
Note that \refass{ass:Z} does not hold for all instances of \refprob{prob:PrimalProblem}. The sets in \reffig{fig:localConv} and \reffig{fig:KKTpoints}, as well as the numerical examples in~\refsec{sec:numerical} satisfy \refass{ass:Z}, whereas \reffig{fig:ZCounterExample} illustrates a simple counter-example in three dimensions.
Moreover, when \refass{ass:Z} is violated, \refalg{alg:krasnoselskij} may still converge. If it converges, the solution is guaranteed to be a local minimum independently of whether \refass{ass:Z} holds, as stated in \refthm{thm:localOptimality}.
%It is not clear whether \refass{ass:Z} holds for all problems of the form of \refprob{prob:MPCProblem}.

In \refalg{alg:krasnoselskij}, we apply the Krasnoselskij iteration to the operator $T_\xi$. This iteration is known to converge \emph{globally} when the operator $T_\xi$ is nonexpansive, which is the case when $\set{Z}$ is \emph{convex}.
To show \emph{local} convergence for the \emph{non-convex} case in \refthm{thm:convergence}, we need $T_\xi$ to be nonexpansive in a neighborhood around fixed points. Starting in such a neighborhood, we will converge via the same argument. \refthm{thm:convergence} is proved using two auxiliary lemmas. 
In \reflem{lem:niceproj} we show that, in a neighborhood around almost all fixed points, the projection $\euclproj{\set{Z}}$ behaves like a projection onto a convex set. Therefore, it is locally firmly nonexpansive, single-valued and continuous \cite[Proposition~4.8, p.~61]{bauschke2011}.
Then, in \reflem{lem:niceop}, we use this property, together with the particular construction of the operator $T_\xi,$ to show that $T_\xi$ is nonexpansive, single-valued and continuous in a neighborhood around almost all fixed points. This is used in \refthm{thm:convergence} to guarantee the \emph{local convergence} of \refalg{alg:krasnoselskij} to a \emph{local minimum}, almost always.
\begin{lemma} \label{lem:niceproj}
	Let Assumptions \refass{ass:existence}--\refass{ass:Z} hold. For any $\xi$ large enough, and any local minimum $\loc{z}$ to \refprob{prob:PrimalProblem} with $\loc{z} \not\in \fix T_\xi$ and $\fixc{T_\xi}{\loc{z}} \neq \varnothing$; we have that for all fixed points $\loc{s} \in \fixc{T_\xi}{\loc{z}}$, except a measure zero subset, there exists an $\epsilon > 0$ such that the projection $\euclproj{\set{Z}}$ is single-valued, continuous and firmly nonexpansive on $\ball_{\epsilon}(\loc{s})$.
\end{lemma}
Note that \reflem{lem:niceproj} is weaker than prox-regularity of $\set{Z}$. Prox-regularity of $\set{Z}$ is equivalent to $\euclproj{\set{Z}}$ being single-valued, Lipschitz-continuous and monotone in a neighborhood around any point $z \in \set{Z}$, see \cite[Theorem~1.3~(i)--(k), p.~5234]{poliquin2000}. \reflem{lem:niceproj} makes a similar statement about $\euclproj{\set{Z}}$ for points in the neighborhood $\ball_{\epsilon}(\loc{s})$ of $\loc{s} \in \fixc{T_\xi}{\loc{z}}$, with $\epsilon > 0$. Because \reflem{lem:niceproj} only holds when $\loc{z} \not\in \fix T_\xi$, the neighborhood $\ball_{\epsilon}(\loc{s})$ can always be chosen such that it does not intersect with $\set{Z}$. In fact, \reffig{fig:proxregularity} illustrates a simple example, where Assumptions~\refass{ass:existence}--\refass{ass:Z}, and thereby \reflem{lem:niceproj}, hold for a set $\set{Z}$ that is not prox-regular at $z = 0$.
\begin{figure}
	\centering
	\begin{tikzpicture}[scale=0.5,cap=round]
	%\clip (-3.5,-1.5) rectangle (2.5,4.8);
	%\draw[style=help lines,step=0.5em] (-1.4,-1.4) grid (1.4,1.4);
	\coordinate (xs) at (0,1);
	%\foreach \i in {5,15,25,35,45,55}{
	%	\fill[fill=orange!\i] (xs) circle ({(65-\i)/25});};
	\fill[fill=color2verylight] (0,0) -- (2.25,2.25) arc (45:180:3.182);
	\node[color=color2] at (1.2,2.8) {$\prnormalcone{\set{Z}}(0)$};
	\draw[very thick] (-3,0) -- (0,0) node[midway,below]{$\set{Z}_1$};
	\draw[very thick,color1] (0,0) -- (2,2) node[below=0.1]{$\set{Z}_2$};
	\fill[fill=verylightdarkgreen] (0,0) -- (0,3) arc (90:135:3);
		\draw[lightdarkgreen,very thick] (0,0) -- (0,2.5);
		\draw[lightdarkgreen,very thick,dashed] (0,2.5) -- (0,3);
		\node[darkgreen] at (-0.7,1.65) {$\set{H}$};
		\draw[lightdarkgreen,very thick] (0,0) -- (-1.77,1.77);
		\draw[lightdarkgreen,very thick,dashed] (-1.77,1.77) -- (-2.12,2.12);
	\fill[fill=color3verylight] (0,0) -- (0,-3) arc (270:315:3);
		\node[below,color3dark] at (1.2,-1.9) {$\fixc{T_\xi}{0}$};
		\draw[color3light,very thick] (0,0) -- (0,-2.5);
		\draw[color3light,very thick,dashed] (0,-2.5) -- (0,-3);
		\draw[color3light,very thick] (0,0) -- (1.77,-1.77);
		\draw[color3light,very thick,dashed] (1.77,-1.77) -- (2.12,-2.12);
	\node at (0,0) [circle,fill=color2,inner sep=1] {}; \node at (0,0.1) [right,color=color2]{$z=0$};
	\end{tikzpicture}
	\caption{The point ${\color{color2}z=0}$ is locally optimal for the problem $\min_{z \in \set{Z}} \tfrac{1}{2}z^\transp z + h^\transp z$, with $\set{Z} := {\color{black}\set{Z}_1} \cup {\color{color1}\set{Z}_2}$, if and only if $h \in {\color{darkgreen}\set{H}} := -\rnormalcone{\set{Z}}(0)$. It can be verified that for all $h \in {\color{darkgreen}\set{H}}$ there exist multipliers $\lambda$ such that $(0,\lambda)$ is a regular KKT point, i.e., \refass{ass:existence} holds. Furthermore, \refass{ass:nondeg} holds for all $\xi > \lambda_{\rm min}^+(R)^{-1}$ and $h \in {\color{darkgreen}\relint\set{H}}$ (almost all $h \in {\color{darkgreen}\set{H}}$), as indicated by \refthm{thm:nondeg}. Also, \refass{ass:Z} holds for $\set{Z}$. Finally, given any $h \in {\color{darkgreen}\set{H}}$: ${\color{color3dark}\fixc{T_\xi}{0}} = \{-h\}$. If $h \neq 0$ ($0 \not\in {\color{color3dark}\fixc{T_\xi}{0}}$), then $\loc{s} := -h \in {\color{color3dark}\fixc{T_\xi}{0}}$ is unique and there exists an $\epsilon > 0$ such that the projection $\euclproj{\set{Z}}(s)$ is like the projection onto the convex set ${\color{color2}\prnormalcone{\set{Z}}(0)}$ for all $s \in \ball_{\epsilon}(\loc{s})$, i.e., it is single-valued, continuous and firmly nonexpansive on $\ball_{\epsilon}(\loc{s})$, as indicated by \reflem{lem:niceproj}.} \label{fig:proxregularity}
\end{figure}
\begin{lemma} \label{lem:niceop}	
	Let $\xi>0$ and given any fixed point $\loc{s} \in \fix T_\xi$ such that there exists an $\epsilon > 0$, where the projection $\euclproj{\set{Z}}$ is single-valued, continuous and firmly nonexpansive on $\ball_{\epsilon}(\loc{s})$. Then the operator $T_{\xi}$ is also single-valued, continuous and nonexpansive on $\ball_{\epsilon}(\loc{s})$.
\end{lemma}
\begin{theorem} \label{thm:convergence}
	Let Assumptions \refass{ass:existence}--\refass{ass:Z} hold. For any $\xi$ large enough, and an initial iterate $s_0$ sufficiently close to a fixed point of $T_{\xi}$, \refalg{alg:krasnoselskij} converges almost always to a fixed point $(\loc{z},\loc{y},\loc{s})$. When it converges, $z_j$ and $y_j$ converge to each other, i.e., $\loc{z}=\loc{y}$, and to a local minimum of \refprob{prob:PrimalProblem}. Moreover, $s_j$ converges to a fixed point of $T_\xi$, i.e., to $\loc{s} \in \fix T_{\xi}$.
\begin{proof}
	Given $\xi > {\lambda_{\rm min}^{+}(R)}^{-1}$, i.e., \eqref{eqn:xiCond}, and large enough, according to \refthm{thm:existence}, \refass{ass:existence} implies the existence of fixed points of $T_\xi$, i.e., $\fix T_{\xi}\neq\varnothing$. We first consider that there exists a fixed point $\opt{s} \in \fix T_\xi$ with corresponding local minimum $\opt{z}$ such that $\opt{s} = \opt{z}$. Then $(\opt{z},0)$ is a KKT point of \refprob{prob:ConsensusProblem}. Therefore
	\begin{equation*}
			M_{\xi}\opt{z}+c_{\xi} \in \euclproj{\set{Z}}(\opt{z}) = \{\opt{z}\} \Leftrightarrow \opt{z} = (I-M_{\xi})^{-1}c_{\xi} = \bar{v}-R(h+H\bar{v})\,,
	\end{equation*}
	with $I-M_{\xi}$ invertible due to $\xi > {\lambda_{\rm min}^{+}(R)}^{-1}$ and \eqref{eqn:Meig}. This is checked in step \ref{alg:stepGlobal} of \refalg{alg:krasnoselskij}. If it holds the algorithm terminates and returns $\opt{z}$, a global minimum.
	If there is no $\loc{s} \in \fix T_\xi$ such that the corresponding local minimum $\loc{z}$ satisfies $\loc{s} = \loc{z}$, we can apply \reflem{lem:niceproj}. It says that for large enough $\xi$ and any local minimum $\loc{z}$, such that $\fixc{T_\xi}{\loc{z}} \neq \varnothing$, we know that for all $\loc{s} \in \fixc{T_\xi}{\loc{z}}$, except a measure zero subset, there exists an $\epsilon>0$ such that the projection $\euclproj{\set{Z}}$ is single-valued, continuous and firmly-nonexpansive on $\ball_{\epsilon}(\loc{s})$. We consider any such $\xi$ and $\loc{s}$. Using \reflem{lem:niceop}, $T_\xi$ is nonexpansive on $\ball_{\epsilon}(\loc{s})$. If the initial iterate is close enough, i.e., $s_0 \in \ball_{\epsilon}(\loc{s})$, then by the nonexpansiveness of $T_{\xi}$ the iterates will satisfy $s_j \in \ball_{\epsilon}(\loc{s})$ for all $j \in \naturals{}$. The Krasnoselskij iteration applied to $T_{\xi}$, in step~\ref{alg:stepKrasnoselskij} of \refalg{alg:krasnoselskij}, converges to a fixed point of $T_{\xi}$ for any choice of step-size parameter $\gamma \in (0,1)$~\cite[Theorem~3.2, p.~65]{berinde2007}. Convergence $\lim_{j\rightarrow\infty}\|z_{j}-y_{j}\| = 0$, follows by single-valuedness, and continuity of $T_{\xi}$. By \refthm{thm:localOptimality}, $z_j$, $y_j$ converge to a local minimum of \refprob{prob:PrimalProblem}. 
\end{proof}
\end{theorem}
\refthm{thm:convergence} may require $\xi$ to be large. Moreover, a lower bound of $\xi$ satisfying Assumption~\refass{ass:nondeg} may be hard to compute. However, in practice good results can be obtained for relatively small values of $\xi$. 
In addition, if the algorithm fails to converge, the theoretical results indicate that  choosing a larger value of $\xi$ may improve the convergence behavior. This is verified experimentally in \refsec{num:xi}, where we examine the behavior for different $\xi$.

\subsection{Non-degeneracy} \label{sec:nondegeneracy}
In the following we show that Assumption~\refass{ass:nondeg} is satisfied for almost all instances of \refprob{prob:PrimalProblem}. \refcor{cor:nondegmpc} makes the result applicable to instances of \refprob{prob:MPCProblem}, covering hybrid MPC.

\begin{theorem}[Non-degeneracy] \label{thm:nondeg}
	Consider \refprob{prob:PrimalProblem} with arbitrary parameter $\theta \in \reals{p}$ and cost matrix $H \succ 0$. Given any point $z \in \set{E} \cap \set{Z}$. For almost all linear cost terms $h$ for which $\fixc{T_\xi}{z} \neq \varnothing$ for some $\xi > 0$, we have that \refass{ass:nondeg} holds at $z$.
	\begin{proof}
		By \refprop{prop:proxNormalsExistencei}, \reflem{lem:opKKT} and the definition of $T_\xi$, there exists a $\bar{\xi} > 0$ such that for all $\xi \geq \bar{\xi}$ we have that $\fixc{T_\xi}{z} \neq \varnothing$ if and only if there exists a $\kkt{\lambda}$ such that $(z,\kkt{\lambda})$ is a regular KKT point. We will show that for all linear cost terms $h$, except a measure zero subset, there exists a $\loc{\lambda}$ such that $\loc{s} := z - \tfrac{1}{\xi}\loc{\lambda} \in \fixc{T_\xi}{z}$ and $-\loc{\lambda} \in \relint \rnormalcone{\set{Z}}(z)$. This means that \refass{ass:nondeg} holds at $z$.
		
		The set $\set{H}$ of $h$ for which there exist $\kkt{\lambda}$ such that $(z,\kkt{\lambda})$ is a regular KKT point is given by the following convex set
		\begin{align*}
			\set{H} :=&\: \big\{ h \in \reals{n} \sep{} \exists \lambda : 0 \in \{Hz + h\} + \rnormalcone{\set{E}}(z) - \{\lambda\} \text{ and} -\lambda \in \rnormalcone{\set{Z}}(z) \big\}\\
			=&\: \begin{bmatrix} I & 0 & 0 \end{bmatrix} \Big( \big( \reals{n} \times -\rnormalcone{\set{Z}}(z) \times \rnormalcone{\set{E}}(z) + \{Hz\} \big) \cap  \big\{\begin{bsmallmatrix}h \\ \lambda \\ v \end{bsmallmatrix} \sep{} v = \lambda - h \big\} \Big)\,.
		\end{align*}
		Due to convexity of the regular normal cones all sets involved in the description of $\set{H}$ are convex. Therefore, we can use distributivity of the $\relint$ operator over the Cartesian product and over the Minkowski sum \cite[Exercise~2.45(a--b), p.~67]{rockafellar1998}, the interchangeability of linear maps and the $\relint$ operator \cite[Proposition~2.44(a), p.~66]{rockafellar1998}, and distributivity of the $\relint$ operator over finite intersections if $\relint \set{H} \neq \varnothing$, see \cite[Proposition~2.42, p.~65]{rockafellar1998}. From this we obtain that
		\begin{align*}
			\relint \set{H} &= \begin{bmatrix} I & 0 & 0 \end{bmatrix} \Big( \big( \reals{n} \times -\relint \rnormalcone{\set{Z}}(z) \times \rnormalcone{\set{E}}(z) + \{Hz\} \big) \cap  \big\{\begin{bsmallmatrix}h \\ \lambda \\ v \end{bsmallmatrix} \sep{} v = \lambda - h \big\} \Big)\\
			&= \big\{ h \in \reals{n} \sep{} \exists \lambda : 0 \in \{Hz + h\} + \rnormalcone{\set{E}}(z) - \{\lambda\} \text{ and } -\!\lambda \in \relint \rnormalcone{\set{Z}}(z) \big\}\,,
		\end{align*}
		if $\relint \set{H} \neq \varnothing$. Therefore, all linear cost terms for which \refass{ass:nondeg} holds are in $\relint \set{H}$. Conversely, via \cite[Theorem.~1, p.~90]{lang1986} and convexity of $\set{H}$, all linear cost terms for which \refass{ass:nondeg} does not hold are in a measure zero subset of $\set{H}$. To conclude the proof, it remains to show that $\relint \set{H} \neq \varnothing$.
		The set $\rnormalcone{\set{Z}}(z)$ is closed, convex and non-empty, this implies that $\relint \rnormalcone{\set{Z}}(z) \neq \varnothing$ via \cite[Proposition~2.40, p.~64]{rockafellar1998}. Therefore, there exists a $\lambda$ such that $-\lambda \in \relint \rnormalcone{\set{Z}}(z)$. Then $h := \lambda - Hz \in \relint \set{H}$, which implies $\relint \set{H} \neq \varnothing$.
	\end{proof}
\end{theorem}

To equivalently transform \refprob{prob:MPCProblem} into \refprob{prob:PrimalProblem}, in \refrem{rem:MPCConsCost}, we have effectively restricted the choice of cost function to the case where its gradient $Hz+h$ is in the set $\set{E}^\alpha$, for $\alpha \in (0,1)^N$, where the set $\set{E}^\alpha$ is defined as
\begin{equation*}
	\set{E}^\alpha := \bigtimes_{k=1}^N\nolimits \Big( \reals{n_u} \times \big\{ (w_k,x_{k+1}) \in \reals{2n_x} \sep{\big} \alpha_k w_k = (1-\alpha_k)x_{k+1} \big\} \Big)\,.
\end{equation*}
With $\set{E}$ as defined for \refprob{prob:MPCProblem} in \refsec{sec:mpc}, we have that $\set{E}^\alpha + \set{E}^\orth = \reals{n}$, i.e. for any $\alpha \in (0,1)^N$: $\set{E}^\alpha$ and $\set{E}^\orth$ span $\reals{n}$.
\refthm{thm:nondeg} does not immediately apply to this case, because the choice of cost function has been restricted. \refcor{cor:nondegmpc} extends \refthm{thm:nondeg} to \refprob{prob:MPCProblem}.
\begin{corollary}[Non-degeneracy of \refprob{prob:MPCProblem}] \label{cor:nondegmpc}
	Consider \refprob{prob:MPCProblem} with arbitrary initial state $\theta \in \reals{p}$, $\alpha \in (0,1)^N$ and quadratic cost matrix $H \succ 0$ following \refrem{rem:MPCConsCost}. Given any point $z \in \set{E} \cap \set{Z}$ and for all linear cost terms $h$ satisfying \refrem{rem:MPCConsCost} for which $\fixc{T_\xi}{z} \neq \varnothing$ for some $\xi > 0$, except a measure zero subset of $\set{E}^\alpha$, we have that \refass{ass:nondeg} holds at $z$.
	\begin{proof}[Proof sketch]
		Following \refsec{sec:mpc}, any instance of \refprob{prob:MPCProblem} can be written as an instance of \refprob{prob:PrimalProblem}, where, according to \refrem{rem:MPCConsCost}, we require cost functions such that $Hz+h \in \set{E}^\alpha$, for some $\alpha \in (0,1)^N$ and any $z$. Following the proof of \refthm{thm:nondeg}, we only need to show that $\relint \set{H} \cap \set{E}^\alpha \neq \varnothing$. This then implies that $\relint (\set{H} \cap \set{E}^\alpha) = \relint (\set{H}) \cap \set{E}^\alpha$, which implies the result. As in the proof of \refthm{thm:nondeg}, we have that $\relint \rnormalcone{\set{Z}}(z) \neq \varnothing$. We need to show that there exists an $h \in \set{E}^\alpha$ such that $(\{Hz+h\} + \rnormalcone{\set{E}}(z)) \cap -\relint \rnormalcone{\set{Z}}(z) \neq \varnothing$.
		Due to \cite[Theorem~6.46, p.~231]{rockafellar1998} which details the presentation of the regular normal cones of polyhedral sets, we have that $\rnormalcone{\set{E}}(z) = \set{E}^\orth$.
		Furthermore we have $Hz + \set{E}^\alpha = \set{E}^\alpha$, for any $z$, by definition. Therefore, we have $(\{Hz\} + \set{E}^\alpha + \rnormalcone{\set{E}}(z)) \cap -\relint \rnormalcone{\set{Z}}(z) = \reals{n} \cap -\relint \rnormalcone{\set{Z}}(z) \neq \varnothing$.
	\end{proof}
\end{corollary}

%%%%%%%%%%%%%%%%%     SECTION 5     %%%%%%%%%%%%%%%%%%%%%%%%%%
\section{Numerical results} \label{sec:numerical}
We consider two numerical examples and compare our method with an MIP reformulation based on disjunctions \cite[Section~5]{vielma2015}, solved using CPLEX, Gurobi and MOSEK. We used YALMIP \cite{lofberg2004} as optimization interface.
To ensure a fair comparison with the MIP solvers, we have used the \emph{absolute MIP gap} tolerance $\Delta_{\rm gap}$, which can be specified for all three used solvers CPLEX, Gurobi and MOSEK.
The tolerance $\Delta_{\rm gap}$ allows the solver to terminate as soon a solution with objective $p$ is found that can be certified to be at most $\Delta_{\rm gap}$ from the optimal objective $p^\star$, i.e., $p + \Delta_{\rm gap} \leq p^\star$.
For each optimization problem that is solved as part of Section~5 we perform the following steps to determine a meaningful gap:
\begin{enumerate*}[label=\arabic*)]
    \item Solve the problem with the proposed method, yielding a suboptimal solution $\loc{z}$ with objective value $\loc{p}$;
    \item solve the problem to global optimality by warm-starting the MIP solver using $\loc{z}$. This yields the optimal objective value $\opt{p}$;
    \item set the absolute MIP gap tolerance of the MIP solver to $\Delta_{\rm gap} := \loc{p}-\opt{p}$ and re-solve the MIP problem without warm-starting.
\end{enumerate*}
In all instances where a global solver is used in this section, we have set the absolute MIP gap tolerance as described above.
We further compare to an efficient implementation of non-convex ADMM, similar to \cite{takapoui2017}, where the splitting \eqref{prob:ConsensusProblem} was used. Efficient C-code for the proposed method and ADMM was generated automatically using the developed code generation tool \cite{autolim}, integrating explicit solutions of the individual convex projections using MPT3 \cite{herceg2013b}. For \refexa{exa:bemporad} the resulting binaries are less than $\unit[75]{kB}$. For \refexa{exa:racecars} they are $\unit[490]{kB}$ for prediction horizon $N=10$ and $\unit[550]{kB}$ for $N=20$. The binaries for ADMM are of comparable size. In contrast, the executables for the MIP solvers are in the range of $10$ to $\unit[20]{MB}$. The reported timings are obtained on an Intel Core i7 processor \emph{using a single core} running at \unit[2.8]{GHz} with \unit[8]{GB} of RAM. No warm-starting or re-starting is used and we always start with the initial iterate $s_0 = 0$ ($y_0 = \lambda_0 = 0$ for ADMM). We have used step size parameter $\gamma=0.5$ and consensus tolerance $\epsilon_{\rm tol} = 10^{-3}$. The same tolerance was also used for the termination criterion of ADMM, according to \cite{boyd2011}, as absolute, primal and dual residual tolerance.
Both examples are instances of \refprob{prob:MPCProblem}. They therefore satisfy \refass{ass:nondeg} for almost any linear cost term, according to \refcor{cor:nondegmpc}. Furthermore, Assumption~\refass{ass:Z} was verified to hold for both examples using \refalg{alg:assZCheck} in \refapp{app:assCheck}. In fact, for Example~5.1, Assumption~(A3) was verified analytically, whereas for Example~5.2 Assumption~(A3) was verified computationally in one hour and $18$ minutes via an implementation of Algorithm~3 in MATLAB. The implementation was not optimized for speed and it is likely that the computation times could be improved substantially with a more efficient implementation.

Additionally, we have applied the proposed method to suitable problems from the \texttt{MacMPEC} collection \cite{leyffer2009}. This is discussed in \refapp{app:macmpec}.

\subsection{Simple hybrid MPC} \label{exa:bemporad}
We consider a simple example taken from the hybrid MPC literature \cite[Example~4.1, p.~415]{bemporad1999}. It consists of two states, one input and PWA dynamics defined over two regions:
\begin{align*}
	x_{k+1} &= \begin{cases}A_1x_k + Bu_k & \text{if } \begin{bsmallmatrix} 1 & 0\end{bsmallmatrix} x_{k}  \geq 0 \\
	          A_2 x_k + Bu_k & \text{if } \begin{bsmallmatrix} 1 & 0\end{bsmallmatrix} x_{k}  \leq 0 \end{cases}\,, 
\end{align*}
with $u_{k} \in [-1,1]$, a regulation objective $\sum_{k=1}^N \|x_{k+1}\|^2 + \|u_k\|^2$ and
\begin{equation*}
	A_1 := \tfrac{2}{5}\begin{bsmallmatrix}1 & -\sqrt{3} \\ \sqrt{3} & 1\end{bsmallmatrix}\,,\: A_2 := \tfrac{2}{5}\begin{bsmallmatrix}1 & \sqrt{3} \\ -\sqrt{3} & 1\end{bsmallmatrix}\,,\:B := \begin{bsmallmatrix}0\\ 1\end{bsmallmatrix}\,.
\end{equation*}
\begin{figure}
	\begin{minipage}[b]{0.35\linewidth}
		\begin{subfigure}[b]{\linewidth}
		\centering
		\includegraphics[width=\linewidth]{bemporad_consensus.tikz}
		\caption{Consensus violation $\|z_j-y_j\|$ (\textbf{\color{blue}solid}) for $N=40$ at the first time step, with threshold (\textbf{\color{white!40!black}dashed}).} \label{fig:bemporadConsensus}
		\end{subfigure}
		\\
		\begin{subfigure}[b]{\linewidth}
		\centering
		\includegraphics[width=\linewidth]{bemporad_trajectory.tikz}
		\caption{Closed-loop trajectory, for $N=40$ and $10$ time steps. Optimal (\textbf{dashed}), and our method (\textbf{\color{color2}solid}).} \label{fig:bemporadClosedloop}
		\end{subfigure}
	\end{minipage}
	\hspace*{0.015\linewidth}
	\begin{minipage}[b]{0.62\linewidth}
		\begin{subfigure}[b]{\linewidth}
		\centering
		\includegraphics[width=\linewidth]{bemporad_runtime_vielma3.tikz}
		\caption{Solution time in $[s]$ averaged over $10$ closed-loop time steps, for different control horizons. Our method (\textbf{\color{color2}solid, $\triangle$}), ADMM (\textbf{\color{blue}solid, \scalebox{1.5}{$*$}}) and a MIP reformulation solved using CPLEX (\textbf{\color{cplexcolor}dashed, $\diamondsuit$}), Gurobi (\textbf{\color{gurobicolor}dash-dotted, $\Circle$}) and MOSEK (\textbf{\color{mosekcolor}dotted, $\square$}).}  \label{fig:bemporadRuntime}
		\end{subfigure}\\
		\begin{subfigure}[b]{\linewidth}
		\centering
		\includegraphics[width=\linewidth]{bemporad_runtime_prox_vs_admm_vielma3.tikz}
		\caption{Minimum, maximum and median solution time in $[s]$ over $10$ closed-loop time steps for our method (\textbf{\color{color2}solid}) with different control horizons. In comparison to the median solution time of ADMM (\textbf{\color{blue}dashed}) and $5$-th to $95$-th percentile (\textbf{\color{color1light}very light}).}  \label{fig:bemporadRuntimeProxvsADMM}
		\end{subfigure}
	\end{minipage}
	\caption{Comparison of \refexa{exa:bemporad} with initial state $x_1=x_2=1$. }
\end{figure}

For prediction horizons $N=5,10,20,\ldots,60$, we apply the hybrid MPC in closed-loop for $10$ steps. Each MPC problem is solved using the proposed method, with $\xi = 10$, and using ADMM with penalty parameter $\rho = 10$, where the tolerance $\epsilon_{\rm tol} = 10^{-3}$ is used for both methods.
The convergence of our method is illustrated in \reffig{fig:bemporadConsensus} by a decreasing consensus violation $\|z_j-y_j\|$, for $N=40$ and the first time step of the receding horizon problem. Even though not shown, the method also converges for all successive time steps (and horizons $N$), with similar convergence characteristics. This leads to a stable closed-loop trajectory, illustrated in \reffig{fig:bemporadClosedloop}. The trajectory is almost indistinguishable from the optimum, obtained using a MIP reformulation based on disjunctions \cite[Section~5]{vielma2015}.
The computational advantage of our method is underlined by \reffig{fig:bemporadRuntime}, where the average runtime for different control horizons $N$ is shown.
Our method is approximately two orders of magnitude faster than the considered commercial MIP solvers.
Furthermore, \reffig{fig:bemporadRuntimeProxvsADMM} illustrates, that our method is slightly faster than ADMM in terms of the median runtime. The average runtime of ADMM in \reffig{fig:bemporadRuntime} is substantially higher than for our method, because ADMM fails to converge for two of the ten time steps and terminates when it reaches \num{10000} iterations. We were not able to achieve convergence of ADMM by adjusting $\rho$ for these cases.

\subsubsection{Local convergence to local minima}
\begin{figure}
	\centering
	\begin{minipage}[b]{\linewidth}
		\centering
		\begin{subfigure}[b]{\linewidth}
			\centering
			\includegraphics[width=0.7\linewidth]{bemporad_init_objf1.tikz}
			\caption{$\xi = 10$. The method converges in $91.4\%$ of cases and does not converge in $8.6\%$ of cases. Clusters with objective values \textbf{\color{blue}(a)} in $[0.4189, 0.4225]$ ($67.9\%$), \textbf{\color{mosekcolor}(c)} in $[0.9411, 0.9461]$ ($22.9\%$) and \textbf{\color{cplexcolor}(d)} in $[1.5488, 1.5572]$ ($0.6\%$). Cluster \textbf{\color{gurobicolor}(b)} is not yet present, likely because $\xi$ is too low.}
		\end{subfigure}\\
		\begin{subfigure}[b]{\linewidth}
			\centering
			\includegraphics[width=0.7\linewidth]{bemporad_init_objf2.tikz}
			\caption{$\xi = 100$. The method converges in $99.1\%$ of cases and does not converge in $0.9\%$ of cases. Clusters with objective values \textbf{\color{blue}(a)} in $[0.4189, 0.4225]$ ($62\%$), \textbf{\color{gurobicolor}(b)} in $[0.5072, 0.5078]$ ($15.2\%$), \textbf{\color{mosekcolor}(c)} in $[0.9411, 0.9748]$ ($21\%$) and \textbf{\color{cplexcolor}(d)} in $[1.5488, 1.5572]$ ($0.9\%$).}
		\end{subfigure}\\
		\begin{subfigure}[b]{\linewidth}
			\centering
			\includegraphics[width=0.7\linewidth]{bemporad_init_objf3.tikz}
			\caption{$\xi = 1000$. The method converges in $99.5\%$ of cases and does not converge in $0.5\%$ of cases. Clusters with objective values \textbf{\color{blue}(a)} in $[0.4189, 0.4225]$ ($62\%$), \textbf{\color{gurobicolor}(b)} in $[0.5072, 0.5078]$ ($15.6\%$), \textbf{\color{mosekcolor}(c)} in $[0.9411, 0.9748]$ ($21\%$) and \textbf{\color{cplexcolor}(d)} in $[1.5488, 1.5572]$ ($0.9\%$).}
		\end{subfigure}
		\caption{Percentage of initial iterates $s_0$ achieving different objective values for different values of $\xi$. Converged solutions are clustered according to objective values, where each cluster contains multiple local optima. Cluster \textbf{\color{blue}(a)} includes the {\color{color2}optimal objective 0.4189}.} \label{fig:bemporad:obj}
	\end{minipage}\\
	\begin{minipage}[b]{\linewidth}
		\centering
		\begin{subfigure}[b]{0.24\linewidth}
			\centering
			\includegraphics[width=\linewidth]{bemporad_init_eprox_c1.tikz}
			\caption{} \label{fig:bemporad:c1}
		\end{subfigure}
		\hspace*{0.01\textwidth}
		\begin{subfigure}[b]{0.22\linewidth}
			\centering
			\includegraphics[width=\linewidth]{bemporad_init_eprox_c2.tikz}
			\caption{} \label{fig:bemporad:c2}
		\end{subfigure}
		\hspace*{0.01\textwidth}
		\begin{subfigure}[b]{0.22\linewidth}
			\centering
			\includegraphics[width=\linewidth]{bemporad_init_eprox_c3.tikz}
			\caption{} \label{fig:bemporad:c3}
		\end{subfigure}
		\hspace*{0.01\textwidth}
		\begin{subfigure}[b]{0.22\linewidth}
			\centering
			\includegraphics[width=\linewidth]{bemporad_init_eprox_c4.tikz}
			\caption{} \label{fig:bemporad:c4}
		\end{subfigure}
		\vspace*{-1em}
		\caption{Trajectories resulting from different initial iterates $s_0$, for $\xi=100$, in clusters \textbf{\color{blue}(a)}--\textbf{\color{cplexcolor}(d)}. The number of trajectories differing (in terms of the objective) by at most $10^{-5}$ is \textbf{\color{blue}(a)} $7$, \textbf{\color{gurobicolor}(b)} $4$, \textbf{\color{mosekcolor}(c)} $15$ and \textbf{\color{cplexcolor}(d)} $8$.} \label{fig:bemporad}
	\end{minipage}
\end{figure}
In order to illustrate the local convergence behavior of \refalg{alg:krasnoselskij} and the dependence of the method on both the initial iterate $s_0$ and the scaling $\xi$, we have solved \refexa{exa:bemporad} for $N=10$ and \num{50000} different initial iterates $s_0 := z_0 - \tfrac{1}{\xi}\lambda_0$, where $z_0$ was drawn uniformly randomly from $[-1,1]^{50}$ and $\lambda_0$ from $[-10,10]^{50}$. We considered three different proximal scalings $\xi=\{10,100,1000\}$ and the consensus tolerance was $\epsilon_{\rm tol} = 10^{-8}$.
The convergence to local minima can be seen in \reffig{fig:bemporad:obj}, where we have clustered the solutions into four clusters \textbf{\color{blue}(a)}--\textbf{\color{cplexcolor}(d)} and illustrate the percentage of initial iterates achieving certain objective values which correspond to local optima.
In \reffig{fig:bemporad:obj}(a), we can see that for $\xi=10$ we converge to fewer different local solutions, compared to $\xi=100$ or $\xi=1000$ given in \reffig{fig:bemporad:obj}(b)--(c). In particular, the cluster \textbf{\color{gurobicolor}(b)} is absent from \reffig{fig:bemporad:obj}(a) and the cluster \textbf{\color{mosekcolor}(c)} contains a more narrow range of objective values. Furthermore, comparing Figures~\ref{fig:bemporad:obj}(a)--(c), we see that the method converges more frequently for larger $\xi$.
Since the proposed method is a local method, the initial iterate $s_0$ of \refalg{alg:krasnoselskij} has a much large effect than $\xi$ on which local minima our method converges to. However, if $\xi$ is too small, some local minima will cease to correspond to fixed points of the method and the number of cases where the method fails to converge increases.
Moreover, the state trajectories of the solutions for each cluster, for $\xi=100$, are plotted in Figures~\ref{fig:bemporad:c1}--\ref{fig:bemporad:c4}. Each cluster contains multiple local minima that have similar objective value. The \num{50000} different initial iterates only lead to $34$ trajectories that each differ (in terms of the objective value) by at least $10^{-5}$. Cluster \textbf{\color{blue}(a)} with $62\%-68\%$ of initial iterates contains solutions that are very close to the global optimum. The initial iterate $s_0=0$ used in \refexa{exa:bemporad} belongs to this cluster. This illustrates that the presented method is indeed a local method, and does not necessarily converge to the global optimum.

\subsection{Racing} \label{exa:racecars}
In this numerical example, we demonstrate the properties of the proposed algorithm on a more complex problem. To this end, we consider a hybrid MPC problem for racing miniature cars~\cite{liniger2015}, where the friction forces acting on the tires are modeled as a PWA function, see~\cite[p.~108ff]{hempel2015c}. The forward velocity $v_x$ of the car is fixed to $\unitfrac[2]{m}{s}$. The state $x = ( v_y, \omega )$ of the system consists of the lateral velocity $v_y$ and the turning rate $\omega$. The steering angle $\delta$ is the only input to the system. The continuous-time dynamics are described by
\begin{align*}
	\dot{v}_y &= \tfrac{1}{m}(F_{f,y}(v_y,\omega,\delta)+F_{r,y}(v_y,\omega)-mv_x\omega)\,,\\
	\dot{\omega} &= \tfrac{1}{I_Z}(l_fF_{f,y}(v_y,\omega,\delta)-l_rF_{r,y}(v_y,\omega))\,,
\end{align*}
where $m,I_Z,l_f,l_r$ are known model parameters and $F_{f,y}$, $F_{r,y}$ are the lateral friction forces acting on the front and rear tires of the vehicle, respectively. 
They are given as PWA functions with $5$ pieces each, giving rise to an irredundant description of the dynamics with $19$ regions per time step. For a prediction horizon $N$, this leads to $19^N$ possible combinations overall. The dynamics are discretized with a sampling time of $T_s = \unit[20]{ms}$.
Additionally, we impose input constraints $\delta \in [\unit[-23]{{}^\circ},\unit[23]{{}^\circ}]$ and state constraints:
\begin{equation}
	v_y \in [\unitfrac[-1]{m}{s}, \unitfrac[1]{m}{s}]\,, \quad \omega \in [\unitfrac[-8]{rad}{s},\unitfrac[8]{rad}{s}]\,.\label{eqn:racecars:stateconst}
\end{equation}
The objective $\sum_{k=1}^N \|\diag(1,\sqrt{10})(x_{k+1}-\bar{x}_{k+1})\|^2 + \|\delta_{k}-\bar{\delta}_{k}\|^2$, tracks a state $\bar{x}$ and input $\bar{\delta}$ reference trajectory producing an S-shaped motion of the race car.
%Assumption \refass{ass:Z} was verified to hold numerically, using \refalg{alg:assZCheck} in \refapp{app:assCheck}.

%\subsubsection{Closed-loop comparison}
We consider closed-loop behavior, where the MPC is applied in a receding-horizon fashion. For every time step a problem with different initial state and reference is solved and only the first input $u_1$ is applied to the dynamical system. A new problem with updated initial state and reference is solved for the next time step. We compare the closed-loop evolution for the different methods over 150 steps. The initial state is $v_y = \omega = 0$. The dynamics of the system are with respect to the lateral and angular velocities, $v_y$ and $\omega$. 
The MPC problem instances are solved for a prediction horizon of $N=10$ ($\xi = 300$) and $N=20$ ($\xi=400$). 
The penalty parameter of ADMM was chosen to be $\rho=350$ and $\rho=450$ for $N=10$ and $N=20$, respectively. 
A time limit of $\unit[3]{s}$ for $N=10$, and $\unit[43]{s}$ for $N=20$ was imposed on all solvers. Both time limits are well above the computation times achieved by our algorithm, with a median runtime of $\unit[34]{ms}$ for $N=10$, and $\unit[87]{ms}$ for $N=20$, as reported in \reffig{fig:racecars:runtime} and \reftab{tab:racecars:suboptimality}.
For $N=20$, Gurobi has a median runtime of $\unit[10]{s}$, reaching the time limit in a few cases, CPLEX has a median runtime if $\unit[12]{s}$, reaching the time limit in some cases and MOSEK reaches the time limit for most time steps leading to median runtime of to $\unit[43]{s}$.
The runtime of ADMM is similar to the proposed method (median runtime of $\unit[55]{ms}$ for $N=20$), however, it fails to converge in a few cases, as shown in \reffig{fig:racecars:runtime}. In contrast, our method converges in at most 1878 iterations for each of the 150 problems, solved in the closed-loop simulation. This is illustrated in \reffig{fig:racecars:consensus}.

\begin{figure}[ht]
	\centering
	\begin{minipage}[b]{0.35\linewidth}
		\begin{minipage}[b]{\linewidth}
			\centering
			\includegraphics[width=\linewidth]{racecars_consensus.tikz}
			\caption{Consensus violations $\|z_j-y_j\|$ for $N=20$ and time steps $1$ to $150$, illustrating convergence.} \label{fig:racecars:consensus}
		\end{minipage}
		\\
		\begin{minipage}[b]{\linewidth}
			\centering
			\includegraphics[width=\linewidth]{racecars_runtime_vielma3.tikz}
			\caption{Solution time in $[s]$ for time steps $k$ of the closed-loop simulation, for \mbox{$N=20$}, time limit $\unit[43]{s}$ (\textbf{dashed}). Our method (\textbf{\color{color2}solid}), ADMM (\textbf{\color{blue}solid}) and Gurobi (\textbf{\color{gurobicolor}dash-dotted}).} \label{fig:racecars:runtime}
		\end{minipage}
	\end{minipage}
	\hspace*{0.013\linewidth}
	\begin{minipage}[b]{0.62\linewidth}
		\centering
		\begin{subfigure}[b]{0.5\linewidth}
			\centering
			\includegraphics[width=\linewidth]{racecarsN10_trajectory_vielma3.tikz}
			\caption{$N=10$} \label{fig:racecars:N10:trajectory}
		\end{subfigure}
		\hspace*{0.005\columnwidth}
		\begin{subfigure}[b]{0.46\linewidth}
			\centering
			\includegraphics[width=\linewidth]{racecars_trajectory_vielma3.tikz}
			\caption{$N=20$} \label{fig:racecars:N20:trajectory}
		\end{subfigure}
		\caption{Closed-loop position trajectory with initial condition $x = y = \varphi = 0$. Optimal closed-loop trajectory (\textbf{dashed}), our method (\textbf{\color{color2}solid}), ADMM (\textbf{\color{blue}solid}), CPLEX (\textbf{\color{cplexcolor}dashed}), Gurobi (\textbf{\color{gurobicolor}dash-dotted}) and MOSEK ({\color{mosekcolor}dotted}).} \label{fig:racecars:trajectory}
		\vspace*{4.75em}
	\end{minipage}
	\vspace*{-1em}
\end{figure}

\begin{table}
	\centering
	\input{racecars_tableDistComp_vielma3.tex}
	\caption{Comparison of relative 2-norm distance $\frac{\|z-\opt{z}\|}{\|\opt{z}\|}$ to the optimal closed-loop position trajectory $\opt{z}$ for $N=10$ and $N=20$. Additionally, in brackets, we report the median computation times for which these trajectories were achieved.} \label{tab:racecars:suboptimality}
\end{table}
In \reffig{fig:racecars:trajectory}, we compare the position of the car, using the \emph{position} dynamics to transform the velocities into positions, via integration of $\dot{\varphi} = \omega$, $\dot{x} = v_x\cos(\varphi)-v_y\sin(\varphi)$ and $\dot{y} = v_x\sin(\varphi)+v_y\cos(\varphi)$. The proposed method (\textbf{\color{color2}solid}) visibly outperforms the commercial MIP solvers in terms of solution quality. This is also illustrated in \reftab{tab:racecars:suboptimality}, where the relative distance to the optimal closed-loop trajectory is reported.
Our method comes very close to the optimal trajectory, while requiring orders of magnitude less runtime.
Both our method and ADMM provide near optimal solutions for similar computational effort. However, our method additionally provides guarantees on the convergence properties of the algorithm.

\subsubsection{Effect of proximal scaling} \label{num:xi}
To better understand the proximal scaling~$\xi$, we consider $2000$ feasible, random instances with different initial states $(v_y,\omega)$ satisfying~\eqref{eqn:racecars:stateconst}. We solve the MPC problem for $40$ values of~$\xi$, with a horizon of~$N=10$ and step size $\gamma = 0.95$. 
\reffig{fig:racecars:xi} demonstrates the relationship between $\xi$,
\begin{enumerate*}[label=(\roman*)]
	\item the number of solved instances and
	\item the number of iterations needed for convergence.
\end{enumerate*}
For larger~$\xi$ more problems are solved, with only $4$~unsolved problems remaining for~$\xi=5000$. Furthermore, the number of iterations needed to reach the consensus threshold (right axis) grows only modestly with larger~$\xi$.
In all instances, the constraint violations and relative suboptimality do not change significantly. These findings are consistent with our expectations from the theory, i.e., that larger~$\xi$ usually lead to more problems solved but at the expense of slower convergence.
\begin{figure}[ht]
	\centering
	\includegraphics[width=0.6\linewidth]{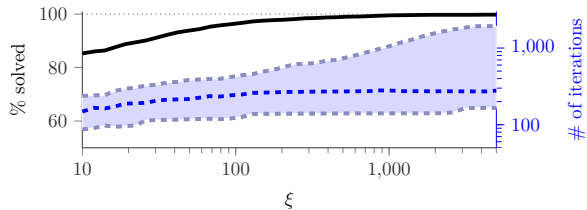}
	\caption{Effect of proximal scaling $\xi$. Percentage of problems solved (\textbf{\color{black}solid}, left axis) and number of iterations to reach $10^{-3}$ consensus violation (right axis) for the solved problems, median (\textbf{\color{blue}dashed}), $5$-th and $95$-th percentile (\textbf{\color{cplexcolor}dashed}).} \label{fig:racecars:xi}
\end{figure}

\section{Conclusion}
We propose a low-complexity method for finding local minima of non-convex, non-smooth optimization problems. This simple and fast method is ideal for hybrid MPC on embedded platforms. In numerical experiments, we observe that our method provides ``good'' locally optimal solutions at a fraction of the time needed by global solvers. Moreover, it is competitive with ADMM in terms of speed. In contrast to other local methods, our algorithm additionally provides local optimality and local convergence guarantees.

%\section*{Acknowledgments}
%We would like to thank the associate editor and reviewers for their time and helpful suggestions.

\appendix

\section{Auxiliary results and proofs} \label{app:proofs}

\subsection{Preliminaries}
We first introduce preliminary results used in this appendix, including the some additional notation:
The distance of a point $x$ to a closed set $\set{C}$ is $\dist(x,\set{C}) := \min_{z \in \set{C}} \|x-z\|$.
The \emph{relative boundary} of a set $\set{C}$ is $\relbd \set{C} := \closure\set{C} \setminus \relint\set{C}$, where $\closure\set{C}$ is the \emph{closure} and $\relint\set{C}$ the \emph{relative interior}.

Moreover, we will make frequent use of the following facts relating the Euclidean projection and the regular normal cone of non-convex polyhedral sets, and the preceding properties of operators in the remainder of this appendix.
\begin{proposition}[Projection and regular normals, \protect{\cite[Example~6.16, p.212]{rockafellar1998}, \cite[Proposition~6.17, p.213]{rockafellar1998}}] \label{prop:proj}
	Given a set $\set{C} \subseteq \reals{n}$.
	\begin{enumerate}[label=(\roman*)]
		\item\label{fct:proxNormal} For any $x \in \set{C}$ it holds that $x \in \euclproj{\set{C}}(x-v) \Rightarrow v \in \rnormalcone{\set{C}}(x)$.
		\item\label{fct:projNormal} If $\set{C}$ is convex, then for any $x \in \set{C}$ it holds that $x \in \euclproj{\set{C}}(x-v) \Leftrightarrow v \in \rnormalcone{\set{C}}(x)$.
		\item\label{fct:proxSingleton} For any $x \in \set{C}$  if $x \in \euclproj{\set{C}}(x-\overbar{\tau} v)$ for some $\overbar{\tau} > 0$, then $\euclproj{\set{C}}(x-\tau v) = \{x\}$ for every $\tau \in (0,\overbar{\tau})$.
	\end{enumerate}
\end{proposition}

\begin{definition}[Nonexpansive and firmly nonexpansive operators, \protect{\cite[Definition~4.1~(i)--(ii), p.~59]{bauschke2011}}] \label{def:nonexp}
	Given an operator $T : \set{X} \rightrightarrows \set{Y}$ and a point $\bar{x} \in \set{X}$.
	\begin{enumerate}[label=(\roman*)]
		\item $T$ is locally \emph{nonexpansive} around $\bar{x}$ if there exists an $\epsilon > 0$ such that for all $x,y \in \ball_\epsilon(\bar{x})$ and any $u \in T(x)$, $v \in T(y)$ it holds that
		\begin{equation*}
			\|u-v\|^2 \leq \|x-y\|^2\,.
		\end{equation*} 
		\item $T$ is locally \emph{firmly nonexpansive} around $\bar{x}$ if there exists an $\epsilon > 0$ such that for all $x,y \in \ball_\epsilon(\bar{x})$ and any $u \in T(x)$, $v \in T(y)$ it holds that
		\begin{equation*}
			\|u-v\|^2 \leq \langle x-y, u-v \rangle\,.
		\end{equation*} 
	\end{enumerate}
\end{definition}

Finally, we elaborate on the structure of the regular normal cone of $\set{Z}$, where the following representation is used: $\set{Z} = \bigcup_{i=1}^{I} \set{Z}^i$.
We remind the reader that $\activePoly{\set{Z}}(z) := \{i \in \{1,\ldots,I\} \sep{} z \in \set{Z}^i\}$ is the set of active components of $\set{Z}$ at a point $z \in \set{Z}$.
\refprop{prop:normalConeCharact} states, that the regular normal cone of $\set{Z}$ at $z$ can be written as the finite intersection of the normal cones of its convex polyhedral parts $\set{Z}^i$, where only the sets $\set{Z}^i$ that contain $z$, i.e., the active components of $\set{Z}$ at $z$, need to be considered.
\begin{proposition}[\!\!\protect{\cite[p.~59f]{henrion2008}}] \label{prop:normalConeCharact}
	The regular normal cone $\rnormalcone{\set{Z}}$ of $\set{Z}$ evaluated at $z \in \set{Z}$ is given by $\rnormalcone{\set{Z}}(z) = \bigcap_{i \in \activePoly{\set{Z}}(z)} \rnormalcone{\set{Z}^i}(z)$.
\end{proposition}

\subsection{Necessary conditions for local optimality}
We start with a proof sketch for \reflem{lem:KKTExistence}, which states that for every local minimum $\loc{z}$ of \refprob{prob:PrimalProblem} there exists multipliers $\loc{\lambda}$ such that $(\loc{z},\loc{\lambda})$ is a generalized KKT point of \refprob{prob:ConsensusProblem}. 
\begin{proof}[Proof sketch of \reflem{lem:KKTExistence}]
	\refprob{prob:PrimalProblem} is equivalent to the consensus problem, \refprob{prob:ConsensusProblem}, which can be written as an optimization problem with variational inequality constraints (OPVIC) as follows
	\begin{align*} \label{prob:opvic}
		\min_{z,y} &\: \cost{z}\\[-0.5em]
		\suchthat &\: \psi(z,y) := \begin{bmatrix}z-y\\y-z\end{bmatrix} \leq 0\,,\: (z,y) \in \set{E} \times \set{Z}\,,\\
		&\: y \in \Omega := \reals{n}\,,\: \langle 0, y-u \rangle \leq 0 \quad \forall u\in\Omega\,.
	\end{align*}
	Because the map $\psi$ is affine, the set $\set{E} \times \set{Z}$ is polyhedral (non-convex) and the set $\Omega$ is polyhedral convex, the constraint set of this problem satisfies a regularity condition called a local error bound \cite[Theorem~4.3, p.~952]{ye2000} at any feasible point. This implies that the problem is calm everywhere, see \cite[Proposition~4.2, p.~951]{ye2000}. In particular, this means that the problem is calm at local solutions, which together with \cite[Theorem~3.6, p.~950]{ye2000} implies that for every local solution $\loc{z}$ to \refprob{prob:ConsensusProblem} there exists a $\loc{\lambda} \in \reals{n}$ such that $(\loc{z},\loc{z},\loc{\lambda})$ satisfies the generalized KKT conditions \eqref{def:KKTpoints} and therefore $(\loc{z},\loc{z},\loc{\lambda})$ is a generalized KKT point.
\end{proof}

\subsection[Proximal KKT points and zeros of K]{Proximal KKT points and zeros of $K_\xi$}
Next, we establishes a necessary and sufficient relationship between regular and $\xi$-proximal KKT points in the proof of \refprop{prop:proxNormalsExistence} and relate the zeros of the operator $K_\xi$ to $\xi$-proximal KKT points of \refprob{prob:ConsensusProblem}, in the proof of \reflem{lem:opKKT}.
\begin{proof}[Proof of \refprop{prop:proxNormalsExistence}]
	\begin{subequations}
	We consider the two statements separately. \begin{enumerate}[label=\roman*),wide]
	\item Any regular KKT point $(z,\lambda)$ satisfies $-\lambda \in \rnormalcone{\set{Z}}(z)$. Thus, by \refprop{prop:normalConeCharact}
	\begin{equation}
		-\lambda \in \rnormalcone{\set{Z}}(z) \Rightarrow -\lambda \in \hspace*{-0.8em}\bigcap_{i \in \activePoly{\set{Z}}(z)}\hspace*{-0.5em} \rnormalcone{\set{Z}^i}(z)\,.\label{eqn:proxNormalsExistence1}
	\end{equation}
	For any $\xi>0$ this is equivalent to
	\begin{equation}
		\text{\eqref{eqn:proxNormalsExistence1}} \Leftrightarrow -\tfrac{1}{\xi}\lambda \in \rnormalcone{\set{Z}^i}(z) \:\:\forall i \in \activePoly{\set{Z}}(z) \Leftrightarrow  z \in \euclproj{\set{Z}^i}(z-\tfrac{1}{\xi}\lambda) \:\:\forall i \in \activePoly{\set{Z}}(z)\,,\label{eqn:proxNormalsExistence2}
	\end{equation}
	where \eqref{eqn:proxNormalsExistence2} follows from the convexity of $\set{Z}^i$ and the relationship between $\euclproj{\set{Z}^i}$ and $\rnormalcone{\set{Z}^i}$, described in \refprop{prop:proj}\ref{fct:projNormal}. By construction, for any $\epsilon > 0$ small enough, we have $\set{Z} \cap \ball_{\epsilon}(z) = \bigcup_{i \in \activePoly{\set{Z}}(z)} \set{Z}^i \cap \ball_{\epsilon}(z)$ and thus $\text{\eqref{eqn:proxNormalsExistence2}} \Rightarrow z \in \euclproj{\set{Z} \cap \ball_{\epsilon}(z)}(z-\tfrac{1}{\xi}\lambda)$.
	We consider $\bar{\xi} := \tfrac{2}{\epsilon}\|\lambda\|$. For any $\xi \geq \bar{\xi}$ we have $\|(z-\tfrac{1}{\xi}\lambda)-z\| = \tfrac{1}{\xi}\|\lambda\|\leq \tfrac{\epsilon}{2}$ . Thus, for all $x \in \set{Z}$ with $\|z-x\| > \epsilon$ we have $\|(z-\tfrac{1}{\xi}\lambda)-x\| > \tfrac{\epsilon}{2}$. Together with $z \in \euclproj{\set{Z} \cap \ball_{\epsilon}(z)}(z-\tfrac{1}{\xi}\lambda)$, this implies that $z \in \euclproj{\set{Z}}(z-\tfrac{1}{\xi}\lambda)$ for all $\xi \geq \bar{\xi} = \tfrac{2}{\epsilon}\|\lambda\|$, which implies that $(z,\lambda)$ is a $\xi$-proximal KKT point for any $\xi \geq \bar{\xi}$.
	\item Given any $\xi>0$, and $\xi$-proximal KKT point $(z,\lambda)$. By definition $z \in \euclproj{\set{Z}}(z-\tfrac{1}{\xi}\lambda )$ which implies $-\tfrac{1}{\xi}\lambda \in \rnormalcone{\set{Z}}(z)$ via \refprop{prop:proj}\ref{fct:proxNormal}. This is equivalent to $-\lambda \in \rnormalcone{\set{Z}}(z)$, which directly implies that $(z,\lambda)$ is a regular KKT point. 
	\end{enumerate}
	\end{subequations}
\end{proof}

\begin{proof}[Proof of \reflem{lem:opKKT}]
	Given any $\xi>0$, and any $\kkt{s} \in \zero K_{\xi}$, by definition $0 \in \{M_{\xi}\kkt{s} + c_{\xi}\} - \euclproj{\set{Z}}(\kkt{s})$, and thereby $\kkt{z} := M_{\xi}\kkt{s} + c_{\xi} \in \euclproj{\set{Z}}(\kkt{s})$. We define $\kkt{\lambda} := \xi(\kkt{z}-\kkt{s})$ and thus $\kkt{z} \in \euclproj{\set{Z}}(\kkt{z}-\tfrac{1}{\xi}\kkt{\lambda}) \;\Leftrightarrow\;\text{\text{\eqref{eqn:ProximalKKTPoint2}}}$ holds.
	Furthermore, we have $\kkt{z} = M_{\xi}(\kkt{z} - \tfrac{1}{\xi}\kkt{\lambda}) + c_{\xi}$ which, using \eqref{eqn:Mrhoxi}, is equivalent to $\kkt{z} = R\left(\kkt{\lambda}-h-H\bar{v}\right) + \bar{v}$. By strict convexity $\kkt{z}$ is the unique solution to $0 \in \partial_z(\tfrac{1}{2}z^\transp H z + (h-\kkt{\lambda})^\transp z + \charfunc{\set{E}}(z))(\kkt{z})$, which is equivalent to $0 \in \{H\kkt{z}+h-\kkt{\lambda}\}+\rnormalcone{\set{E}}(\kkt{z})$ and therefore \eqref{eqn:ProximalKKTPoint1} holds and $(\kkt{z},\kkt{\lambda})$ is a $\xi$-proximal KKT point.
\end{proof}

\subsection{Local convexification}
\reflem{lem:convNCone} is one of the main mathematical tools that we use to prove the local optimality and local convergence of the method. For any $z \in \set{Z}$, it relates the set non-convex $\set{Z}$ to the closed convex set $\prnormalcone{\set{Z}}(z) + \{z\}$, in a neighborhood of $z$. Specifically, it states that $\set{Z}$ is contained in $\prnormalcone{\set{Z}}(z) + \{z\}$, in a neighborhood of $z$, and that the regular normal cones of the two sets at $z$ are the same.
\begin{proof}[Proof of \reflem{lem:convNCone}]
	We prove the two parts individually.
	\begin{enumerate}[label=\roman*),wide]
		\item Given $\epsilon > 0$, take any $\bar{z} \in \set{Z}$ and $z \in \set{Z} \cap \ball_{\epsilon}(\bar{z})$. Using the definition of the polar cone we have that $z \in \prnormalcone{\set{Z}}(\bar{z}) + \{\bar{z}\}$ if and only if $\langle z-\bar{z}, v\rangle \leq 0$ for all $v\in\rnormalcone{\set{Z}}(\bar{z})$.
		We assume for the sake of contradiction that for all $\epsilon > 0$ there exists a $v \in\rnormalcone{\set{Z}}(\bar{z})$ such that $\langle z-\bar{z},v\rangle>0$. Using \refprop{prop:normalConeCharact}, we have that for any $\epsilon >0$ small enough $v \in\rnormalcone{\set{Z}}(\bar{z})$ is equivalent to $\langle v, u-\bar{z} \rangle \leq 0$ for all $u \in \set{Z}^i$, and all $i \in \activePoly{\set{Z}}(\bar{z})$. However, since $z \in \set{Z}$, there exists an $i \in \activePoly{\set{Z}}(\bar{z})$ such that $z \in \set{Z}^i$, and therefore by choosing $u=z$ we have $\langle v, u-\bar{z} \rangle < \langle z-\bar{z},v\rangle$, a contradiction.
		\item For any $z\in\set{Z}$, we have that $\rnormalcone{\set{Z}}(z) = \pprnormalcone{\set{Z}}(z)$, due to $\rnormalcone{\set{Z}}(z)$ closed and convex, see \cite[Corollary~6.21, p.~216]{rockafellar1998}. From the definition of the polar and regular normal cones, see \cite[p.~203, p.~215]{rockafellar1998}, and convexity of $\prnormalcone{\set{Z}}(z)$ we have that $v \in \rnormalcone{\set{Z}}(z)$ is equivalent to $\langle v,w \rangle \leq 0$ for all $w \in \prnormalcone{\set{Z}}(z)$, which is equivalent to $\langle v,u-z \rangle \leq 0$ for all $u \in \prnormalcone{\set{Z}}(z) + \{z\}$. This means $v \in \rnormalcone{\prnormalcone{\set{Z}}(z)+\{z\}}(z)$ and implies the result.
	\end{enumerate}
\end{proof}

\subsection{Local convergence to local minima}
We are now ready to prove the two central lemmas that lead to the local convergence result in \refthm{thm:convergence}: \reflem{lem:niceproj} states that in a neighborhood around almost all fixed points of $T_\xi$, the projection $\euclproj{\set{Z}}$ behaves like a projection onto a convex set. This is used in the proof of \reflem{lem:niceop} to show that the operator $T_\xi$ is nonexpansive, single-valued and continuous in a neighborhood around almost all fixed points.

To prove \reflem{lem:niceproj}, the following two auxiliary propositions are needed. \refprop{prop:projqne} gives a local characterization of the projection around points $\bar{s}$ where it is unique. It states, that there always exists a neighborhood $\ball_\epsilon(\bar{s})$, with $\epsilon>0$, such that for all points in the neighborhood the projection onto $\set{Z}$ can be restricted to the projection onto just the union of the active components of $\set{Z}$ at $\bar{z}=\euclproj{\set{Z}}(\bar{s})$. Furthermore for all points $s \in \ball_\epsilon(\bar{s})$ in the neighborhood the result $z \in \euclproj{\set{Z}}(s)$ of the projection satisfies the bound $\|z-\bar{z}\| \leq \|s-\bar{s}\|$.
\refprop{prop:xilb} states that the set $\fixc{T_\xi}{\loc{z}}$ of fixed points corresponding to a given local optimum $\loc{z}$ of \refprob{prob:PrimalProblem} either is empty for all $\xi$ or is non-empty for all $\xi$ large enough.
\begin{proposition}[Projection] \label{prop:projqne}
	Given $\bar{z}$ and $\bar{s}$ such that $\euclproj{\set{Z}}(\bar{s}) = \{\bar{z}\}$, then there exists an $\epsilon > 0$ such that for all $s\in\ball_{\epsilon}(\bar{s})$: $\euclproj{\set{Z}}(s) = \euclproj{\set{Y}(\bar{z})}(s) \subseteq \ball_{\|s-\bar{s}\|}(\bar{z})$, where $\set{Y}(z) := \bigcup_{i \in \activePoly{\set{Z}}(z)} \set{Z}^i$.
\end{proposition}
\begin{proof}
	First, we will show that for all $s \in \reals{n}$, we have $\euclproj{\set{Y}(\bar{z})}(s) \subseteq \ball_{\|s-\bar{s}\|}(\bar{z})$.
	For all $i \in \activePoly{\set{Z}}(\bar{z})$, the projection $\euclproj{\set{Z}^i}$ is nonexpansive because $\set{Z}^i$ is closed and convex. Therefore, for any $i \in \activePoly{\set{Z}}(\bar{z})$ and any $s \in \reals{n}$ we have $\|\euclproj{\set{Z}^i}(s)-\euclproj{\set{Z}^i}(\bar{s})\| \leq \|s-\bar{s}\|$.
	Consider any $z \in \euclproj{\set{Y}(\bar{z})}(s)$, since $\bar{z}$ and $\bar{s}$ are such that $\euclproj{\set{Z}}(\bar{s}) = \{\bar{z}\}$ it follows that $\euclproj{\set{Y}(\bar{z})}(\bar{s}) = \euclproj{\set{Z}^i}(\bar{s}) = \{\bar{z}\}$ for all $i \in \activePoly{\set{Z}}(\bar{z})$. In particular for some $i \in \activePoly{\set{Z}}(\bar{z})$ we have
	\begin{equation} \label{eqn:projne5}
		\|z-\bar{z}\| = \|\euclproj{\set{Z}^i}(s)-\euclproj{\set{Z}^i}(\bar{s})\| \leq \|s-\bar{s}\|\,,
	\end{equation}
	which implies $\euclproj{\set{Y}(\bar{z})}(s) \subseteq \ball_{\|s-\bar{s}\|}(\bar{z})$.
	\begin{subequations}
	Second, we show that there exists an $\epsilon > 0$ such that for any $y \in \set{Z} \setminus \set{Y}(\bar{z})$
	\begin{equation}
		\|y-s\| > \|z-s\| \quad \forall s \in \ball_{\epsilon}(\bar{s})\,,z \in \euclproj{\set{Y}(\bar{z})}(s)\,. \label{eqn:projqne2}
	\end{equation}
	We define $\varsigma := \inf_{y \in \set{Z} \setminus \set{Y}(\bar{z})} \|y-\bar{s}\| - \|\bar{z}-\bar{s}\| > 0$, by closedness of $\set{Z}$ and $\set{Z}^i$. The triangle inequality implies that for any $y \in \set{Z} \setminus \set{Y}(\bar{z})$ we have $\|y-\bar{s}\| \leq \|y-s\| + \|s-\bar{s}\|$.
	We choose $\epsilon < \frac{\varsigma}{3}$, which implies
	\begin{equation}
		\|y-s\| \geq \|y-\bar{s}\| - \|s-\bar{s}\| \geq \varsigma + \|\bar{z}-\bar{s}\| - \|s-\bar{s}\| > 3\epsilon + \|\bar{z}-\bar{s}\| - \epsilon\,, \label{eqn:projqne3}
	\end{equation}
	where in the second step we have used the definition of $\varsigma$, and in the third step we have used $\varsigma > 3\epsilon$ and $\|s-\bar{s}\| \leq \epsilon$.
	Now we consider any $z \in \euclproj{\set{Y}(\bar{z})}(s)$, via the triangle inequality $\|z-s\| \leq \|z-\bar{z}\| + \|s-\bar{s}\| + \|\bar{z}-\bar{s}\|$
	and \eqref{eqn:projne5} we have
	\begin{equation*}
		\|z-s\| \leq 2\|s-\bar{s}\| + \|\bar{z}-\bar{s}\| \leq 2\epsilon + \|\bar{z}-\bar{s}\|\,,
	\end{equation*}
	which together with \eqref{eqn:projqne3} implies \eqref{eqn:projqne2}.
	\end{subequations}
	This implies the result.
\end{proof}

\begin{proposition} \label{prop:xilb}
	Given any local minimum $\loc{z}$ of \refprob{prob:PrimalProblem}.\\Either $\fixc{T_\xi}{\loc{z}} = \varnothing$ for all $\xi > 0$ or there exists a $\bar{\xi}>0$ such that $\fixc{T_\xi}{\loc{z}} \neq \varnothing$ for all $\xi \geq \bar{\xi}$.
\begin{proof}
	Given any local minimum $\loc{z}$. Either $\fixc{T_{\xi}}{\loc{z}} = \varnothing$ for all $\xi > 0$ or there exists a $\bar{\xi}$ such that there exists an $\bar{s} \in \fixc{T_{\bar{\xi}}}{\loc{z}} \neq \varnothing$. By \reflem{lem:opKKT}, it holds that $(\loc{z},\lambda)$, with $\lambda := \bar{\xi}(\loc{z}-\bar{s})$, is a $\bar{\xi}$-proximal KKT point. From \refprop{prop:proxNormalsExistence}, it further follows that $(\loc{z},\lambda)$ is also a $\xi$-proximal KKT point for any $\xi \geq \bar{\xi}$, which implies $s := \loc{z} - \tfrac{1}{\xi}\lambda \in \fixc{T_{\xi}}{\loc{z}}$, and thus $\fixc{T_{\xi}}{\loc{z}} \neq \varnothing$ for all $\xi \geq \bar{\xi}$.
\end{proof}
\end{proposition}

We use the following claim to prove \reflem{lem:niceproj}. The proof of \refclm{clm:niceproj} is presented immediately after the proof of \reflem{lem:niceproj}. The claim states that if $\xi$ is chosen large enough, then for any local minimum $\loc{z}$ for which the operator $T_\xi$ has fixed points such that $\loc{z}$ itself is not a fixed point of $T_\xi$, we have that for almost all fixed points of $T_\xi$, the projection is unique and the fixed point lies in the relative interior of the shifted regular normal cone $\rnormalcone{\set{Z}}(\loc{z}) + \{\loc{z}\}$ of $\set{Z}$ at $\loc{z}$.
\begin{claim}\label{clm:niceproj}
	Given Assumptions \refass{ass:existence}--\refass{ass:nondeg}. For any $\xi>0$ large enough and any local minimum $\loc{z}$ to \refprob{prob:ConsensusProblem}, with $\fixc{T_\xi}{\loc{z}} \neq \varnothing$ and $\loc{z} \not\in \fixc{T_\xi}{\loc{z}}$, we have that
	\begin{equation}\label{eqn:niceprojCond}
		\euclproj{\set{Z}}(\loc{s}) = \{\loc{z}\} \text{ and } \loc{s} \in \relint \rnormalcone{\set{Z}}(\loc{z}) + \{\loc{z}\}\,,
	\end{equation}
	for all fixed points $\loc{s} \in \fixc{T_\xi}{\loc{z}}$, except a measure zero subset.
\end{claim}

\begin{proof}[Proof of \reflem{lem:niceproj}]
	Given \refclm{clm:niceproj} and consider any pair $\loc{z},\loc{s}$ satisfying \eqref{eqn:niceprojCond}. We will show that in a neighborhood of $\loc{s}$ the projection onto $\set{Z}$ is the same as projecting onto an appropriate closed, convex set. This then implies local firm nonexpansiveness, continuity and single-valuedness of the projection.
	
	Let $\ball_\epsilon(\loc{s})$ be a neighborhood of $\loc{s}$, for some $\epsilon>0$ small enough. We define $\set{S} := \lin (\rnormalcone{\set{Z}}(\loc{z}))$, the linear subspace spanned by $\rnormalcone{\set{Z}}(\loc{z})$, and its orthogonal complement $\set{S}^\orth$. For any $s \in \ball_\epsilon(\loc{s})$, we can write $s = \loc{s} + u + v$, where $u \in \set{S}$, $v \in \set{S}^\orth$ and $u+v \in \ball_\epsilon$. This allows us to argue about the contribution of $u$ and $v$ successively. We consider $\bar{s}:=\loc{s}+u$. Due to $\euclproj{\set{Z}}(\loc{s}) = \{\loc{z}\}$, continuity of the distance function, and \refprop{prop:proj}\ref{fct:proxSingleton}, we have $\euclproj{\set{Z}}(\bar{s}) = \euclproj{\set{Z}}(\loc{s}+u) = \{\loc{z}\}$ for all $u \in \ball_{\epsilon}$, given that $\epsilon$ is small enough. Furthermore, due to openness of $\relint \rnormalcone{\set{Z}}(\loc{z})$ on $\set{S}$, we have $\bar{s} \in \relint \rnormalcone{\set{Z}}(\loc{z}) +\{\loc{z}\}$ for all $u \in \ball_{\epsilon}$, with $\epsilon>0$ small enough. Therefore instead of considering a pair $\loc{z},\loc{s}$ satisfying \eqref{eqn:niceprojCond}, we will consider $\loc{z},\bar{s}$. We now examine the contribution of $v$, i.e., $s = \bar{s} + v$. Similarly to \reflem{lem:convNCone}, where we showed that we can replace the regular normal cone of $\set{Z}$ by the normal cone of $\prnormalcone{\set{Z}}(\loc{z})+\{z\}$, we will now consider the projection onto that set. We do this because $\prnormalcone{\set{Z}}(\loc{z})+\{\loc{z}\}$ is a local convexification of $\set{Z}$. We have $v \in \set{S}^\orth \subseteq \prnormalcone{\set{Z}}(\loc{z})$. We consider the projection $\euclproj{\prnormalcone{\set{Z}}(\loc{z})}(s-\loc{z}) = \argmin_{w \in \prnormalcone{\set{Z}}(\loc{z})} \|\bar{s}-\loc{z}+v-w\|^2 = \argmin_{w \in \prnormalcone{\set{Z}}(\loc{z})} \|\bar{s}-\loc{z}\|^2 + \|v-w\|^2 - 2 \langle \bar{s}-\loc{z}, w \rangle$, where we have used $\langle \bar{s}-\loc{z}, v \rangle = 0$. By $\langle \bar{s}-\loc{z}, w \rangle \leq 0$, $v$ is a minimizer. It is unique because the set $\prnormalcone{\set{Z}}(\loc{z})$ is closed and convex, therefore $\euclproj{\prnormalcone{\set{Z}}(\loc{z})}(s-\loc{z}) = \{v\}$ and consequently $\euclproj{\prnormalcone{\set{Z}}(\loc{z}) +\{\loc{z}\}}(\bar{s}+v) = \{\loc{z}+v\}$.
	Additionally, due to $\set{Z} \cap \ball_{\epsilon}(\loc{z}) \subseteq\prnormalcone{\set{Z}}(\loc{z}) +\{\loc{z}\}$ from \reflem{lem:convNCone} and $\loc{z}+v \in \set{Z} \cap \ball_{\epsilon}(\loc{z})$ due to Assumption \refass{ass:Z}, it follows that $\loc{z}+v \in \euclproj{\set{Z}}(\bar{s}+v)$.
	Finally using \refprop{prop:projqne} we know that $\euclproj{\set{Z}}(\bar{s}+v) \subseteq \set{Z} \cap \ball_{\|v\|}(\loc{z}) \subseteq \set{Z} \cap \ball_\epsilon(\loc{z})$, which immediately implies that $\euclproj{\set{Z}}(\bar{s}+v) = \euclproj{\prnormalcone{\set{Z}}(\loc{z}) +\{\loc{z}\}}(\bar{s}+v) = \{\loc{z}+v\}$.
	This means for all $s \in \ball_\epsilon(\loc{s})$ the projection onto $\set{Z}$ is equivalent to the projection onto the closed, convex set $\prnormalcone{\set{Z}}(\loc{z}) + \{\loc{z}\}$ and therefore enjoys all of the properties of a projection onto a closed convex set. In particular single-valuedness, continuity and firm nonexpansiveness \cite[Proposition~4.8, p.~61]{bauschke2011}. 
\end{proof}

\begin{proof}[Proof of \refclm{clm:niceproj}]
	Given a local minimum $\loc{z}$ to \refprob{prob:ConsensusProblem} and consider any $\bar{\xi}$ according to \refprop{prop:xilb} such that \refass{ass:nondeg} holds. By definition of $\fix T_\xi$ and by \refprop{prop:proj}\ref{fct:proxSingleton}, \refass{ass:nondeg} also holds for any larger $\xi \geq \bar{\xi}$. We will show that for any $\xi > \bar{\xi}$, and any local minimum $\loc{z}$ such that $\fixc{T_\xi}{\loc{z}} \neq \varnothing$ and $\loc{z} \not\in \fixc{T_\xi}{\loc{z}}$:
	\begin{equation}\label{eqn:relintfix}
		\relint \fixc{T_\xi}{\loc{z}} = \big\{ s \in \fixc{T_\xi}{\loc{z}} \sep{\big} \euclproj{\set{Z}}(s) = \{\loc{z}\}\,,\: s \in \relint \rnormalcone{\set{Z}}(\loc{z}) + \{\loc{z}\} \big\}\,.
	\end{equation}
	\cite[Theorem.~1, p.~90]{lang1986} states that the boundary of a convex set has zero Lebesgue measure. 
	Convexity of $\fixc{T_\xi}{\loc{z}}$ together with \eqref{eqn:relintfix} therefore imply that \eqref{eqn:niceprojCond} holds for all $s \in \fixc{T_\xi}{\loc{z}}$, except a measure zero subset.
	
	Due to $\fixc{T_{\bar{\xi}}}{\loc{z}} \neq \varnothing$ and \refass{ass:nondeg}, we have that there exists an $\bar{s} \in \fixc{T_{\bar{\xi}}}{\loc{z}}$ with $\bar{s} \in \relint \rnormalcone{\set{Z}}(\loc{z}) + \{\loc{z}\}$. Furthermore, we have that $s(\xi) := \loc{z} - \tfrac{1}{\xi}\bar{\xi}(\loc{z}-\bar{s}) \in \fixc{T_\xi}{\loc{z}}$ and $s(\xi) \in \relint \rnormalcone{\set{Z}}(\loc{z}) + \{\loc{z}\}$ for $\xi \geq \bar{\xi}$ and by \refprop{prop:proj}\ref{fct:proxSingleton} we have that $\euclproj{\set{Z}}(s(\xi)) = \{\loc{z}\}$ for any $\xi > \bar{\xi}$. Therefore, for any $\xi > \bar{\xi}$ we have that 
	\begin{equation} \label{eqn:claim:nonempty}
		\fixc{T_\xi}{\loc{z}} \cap \relint \big(\rnormalcone{\set{Z}}(\loc{z}) + \{\loc{z}\}\big) \cap \big\{s\sep{\big}\euclproj{\set{Z}}(s) = \{\loc{z}\}\big\} \neq \varnothing\,.
	\end{equation}
	We consider the set $\{ s \sep{}  \loc{z} \in \euclproj{\set{Z}}(s) \}$. We have
	\begin{align*}
		\big\{ s \sep{\big}  \loc{z} \in \euclproj{\set{Z}}(s) \big\} &= \big\{ s \sep{\big} \|s-\loc{z}\| \leq \dist(s,\set{Z}) \big\} \\
		&= \big\{ s \sep{\big} \|s-\loc{z}\| \leq \dist(s,\set{Y}(\loc{z})) \big\} \cap \big\{ s \sep{\big} \|s-\loc{z}\| \leq \dist(s,\set{Y}^c(\loc{z})) \big\} \\
		&= \big(\rnormalcone{\set{Z}}(\loc{z}) + \{\loc{z}\}\big) \cap \big\{ s \sep{\big} \|s-\loc{z}\| \leq \dist(s,\set{Y}^c(\loc{z})) \big\}\,,
	\end{align*}
	with $\set{Y}(\loc{z}) := \bigcup_{i\in\activePoly{\set{Z}}(\loc{z})} \set{Z}^i$ and $\set{Y}^c(\loc{z}) := \bigcup_{i\in\activePoly{\set{Z}}^c(\loc{z})} \set{Z}^i$, where $\activePoly{\set{Z}}^c(\loc{z}) := \{1,\ldots,I\}\setminus\activePoly{\set{Z}}(\loc{z})$. Where we have used that by \refprop{prop:normalConeCharact} and convexity of $\set{Z}^i$ we have
	\begin{align*}
		\{ s \sep{} \|s-\loc{z}\| \leq \dist(s,\set{Y}(\loc{z})) \} &= \{ s \sep{} \loc{z} \in \euclproj{\set{Y}(\loc{z})}(s) \}\\
		&= \rnormalcone{\set{Y}(\loc{z})}(\loc{z}) + \{\loc{z}\} = \rnormalcone{\set{Z}}(\loc{z}) + \{\loc{z}\}\,.
	\end{align*}
	By continuity and upper-boundedness of $\varsigma(s) := \|s-\loc{z}\| - \dist(s,\set{Y}^c(\loc{z}))$, and by closedness of $\set{Y}^c(\loc{z})$, it follows that
	\begin{align*}
		\relint \big\{ s \sep{\big} \|s-\loc{z}\| \leq \dist(s,\set{Y}^c(\loc{z})) \big\} &= \big\{ s \sep{\big} \|s-\loc{z}\| < \dist(s,\set{Y}^c(\loc{z})) \big\}\\
		&= \big\{ s \sep{\big} \euclproj{\set{Z}}(s) = \{\loc{z}\} \big\}
	\end{align*}
	for all $s$ such that $\loc{z} \in \euclproj{\set{Z}}(s)$. Therefore, by distributivity of the $\relint$ operator over finite intersections due to \eqref{eqn:claim:nonempty}, see \cite[Proposition~2.42, p.~65]{rockafellar1998}, we have
	\begin{align*}
		\relint \big\{ s \sep{\big}  \loc{z} \in \euclproj{\set{Z}}(s) \big\} = \relint \big(\rnormalcone{\set{Z}}(\loc{z}) + \{\loc{z}\}\big) \cap \big\{ s \sep{\big} \euclproj{\set{Z}}(s) = \{\loc{z}\} \big\} \neq \varnothing\,.
	\end{align*}
	By definition of $\fixc{T_\xi}{\loc{z}}$, convexity of $\{ s \sep{}  \loc{z} \in \euclproj{\set{Z}}(s) \}$, \cite[Proposition~2.42, p.~65]{rockafellar1998} and \eqref{eqn:claim:nonempty} we then have
	\begin{align*}
		\relint \fixc{T_\xi}{\loc{z}} &= \relint \big\{ s \sep{\big} 0 \in \{H\loc{z}+h\} + \rnormalcone{\set{E}}(\loc{z}) - \{\xi(\loc{z}-s)\}\,,\: \loc{z} \in \euclproj{\set{Z}}(s) \big\}\\
		&= \fixc{T_\xi}{\loc{z}} \cap \relint \big\{ s \sep{\big}  \loc{z} \in \euclproj{\set{Z}}(s) \big\}\\
		&= \fixc{T_\xi}{\loc{z}} \cap \relint \big(\rnormalcone{\set{Z}}(\loc{z}) + \{\loc{z}\}) \cap \big\{ s \sep{\big} \euclproj{\set{Z}}(s) = \{\loc{z}\} \big\} \neq \varnothing\,.
	\end{align*}
\end{proof}

\begin{proof}[Proof of \reflem{lem:niceop}]
	We consider any $\xi>0$, $\epsilon > 0$ and $\loc{s} \in \fix T_\xi$ such that the projection $\euclproj{\set{Z}}$ is single-valued, continuous and firmly nonexpansive on $\ball_{\epsilon}(\loc{s})$. We recall the definition of $T_\xi$ in \eqref{eqn:T2}.
	Clearly $T_{\xi}$ is single-valued and continuous on $\ball_{\epsilon}(\loc{s})$ by virtue of the projection. It remains to show the nonexpansiveness of $T_{\xi}$ on $\ball_{\epsilon}(\loc{s})$. By \cite[Definition~4.1~(ii), p.~59]{bauschke2011}, $T_{\xi}$ is nonexpansive on $\ball_{\epsilon}(\loc{s})$ if $\|T_{\xi}(s)-T_{\xi}(t)\|^2 \leq \|s-t\|^2$ for all $s,t \in \ball_{\epsilon}(\loc{s})$.
	For any $s,t$, we have that
	\begin{align} \label{eqn:Tdiffnorm}
		\|T_{\xi}(s)-T_{\xi}(t)\|^2 &= \|(I-WM_{\xi})(s-t)\|^2\\
		&\phantom{=}\: + 2\langle W^\transp (I-WM_{\xi})(s-t), \euclproj{\set{Z}}(s)-\euclproj{\set{Z}}(t) \rangle\nonumber\\
		&\phantom{=}\: + \|W(\euclproj{\set{Z}}(s)-\euclproj{\set{Z}}(t))\|^2\,.\nonumber
	\end{align}
	To show $\eqref{eqn:Tdiffnorm} \leq \|s-t\|^2$ for all $s,t \in \ball_{\epsilon}(\loc{s})$ we will split \eqref{eqn:Tdiffnorm} into two parts, the portions belonging to the range- and nullspace of $M_{\xi}$. We recall that $M_{\xi}$ is positive semi-definite for $\xi$ satisfying \eqref{eqn:xiCond}. Therefore, we can consider the eigenvalue decomposition of $M_{\xi}$, as in \eqref{def:fixedpointOp} and define $q_i$, for $i=1,\ldots,n$, to be the normalized eigenvectors of $M_{\xi}$. We define $\delta,p \in \reals{n}$ as coefficients, such that $\sum_{i=1}^n \delta_i q_i = s-t$ and $\sum_{i=1}^n p_i q_i = \euclproj{\set{Z}}(s)-\euclproj{\set{Z}}(t)$.  We further define $\delta^\orth := \sum_{i=1}^{n-m} \delta_i q_i$, $\delta^0 := \sum_{i=n-m+1}^n \delta_i q_i$, $p^\orth := \sum_{i=1}^{n-m} p_i q_i$ and $p^0 := \sum_{i=n-m+1}^n p_i q_i$, which splits $s-t$ and $\euclproj{\set{Z}}(s)-\euclproj{\set{Z}}(t)$ into their range and nullspace portions. Additionally, $p^\orth_{\Lambda} := \sum_{i=1}^{n-m} \lambda_i^{-1}p_iq_i$, for $\lambda_i$ the non-zero eigenvalues of $M_\xi$.
	Substituting $\delta^\orth$, $\delta^0$, $p^\orth$, $p^0$ and $p^\orth_{\Lambda}$ into \eqref{eqn:Tdiffnorm}, using the eigenvalue decomposition of $M_\xi$ and $W$ as defined in \eqref{eqn:W}, we obtain
	\begin{align} \label{eqn:TdiffnormSplit}
		\|T_{\xi}(s)-T_{\xi}(t)\|^2 =\:&\tfrac{1}{4}\|\delta^\orth\|^2 + \|\delta^0\|^2 + \tfrac{1}{2}\langle \delta^\orth,p^\orth_{\Lambda}\rangle - 2\langle \delta^0,p^0\rangle + \tfrac{1}{4}\|p^\orth_{\Lambda}\|^2 + \|p^0\|^2\,.
	\end{align}
	We want to show that $\eqref{eqn:Tdiffnorm} = \eqref{eqn:TdiffnormSplit} \leq \|s-t\|^2 = \|\delta^\orth\|^2 + \|\delta^0\|^2$. However, this only has to hold for $\delta,p$ that are related via the projection $\euclproj{\set{Z}}$ where in particular we want to exploit its firm nonexpansiveness. The projection $\euclproj{\set{Z}}$ being firmly nonexpansive, and nonexpansive, on $\ball_{\epsilon}(\loc{s})$, means via \refdef{def:nonexp}, that for all $s,t \in \ball_{\epsilon}(\loc{s})$
	 \begin{align}
		&\|\euclproj{\set{Z}}(s)-\euclproj{\set{Z}}(t)\|^2 \leq \langle s-t, \euclproj{\set{Z}}(s)-\euclproj{\set{Z}}(t) \rangle\text{, and}\tag{FNE}\label{eqn:ProjFNE}\\
		&\|\euclproj{\set{Z}}(s)-\euclproj{\set{Z}}(t)\|^2 \leq \|s-t\|^2\,.\tag{NE}\label{eqn:ProjNE}
	\end{align}
	By applying the Cauchy-Schwarz inequality, and combining \eqref{eqn:ProjFNE} and \eqref{eqn:ProjNE} we get
	\begin{equation}
		0 \leq \|\euclproj{\set{Z}}(s)-\euclproj{\set{Z}}(t)\|^2 \leq \langle s-t, \euclproj{\set{Z}}(s)-\euclproj{\set{Z}}(t) \rangle \leq \|s-t\|^2\,.\label{eqn:fnecond}
	\end{equation}
	Substituting $\delta^\orth, \delta^0, p^\orth$ and $p^0$ into \eqref{eqn:fnecond}, we obtain
	\begin{equation} \label{eqn:fnecondsplit}
		0 \leq \|p^\orth\|^2 + \|p^0\|^2 \leq \langle p^\orth,\delta^\orth \rangle + \langle p^0,\delta^0 \rangle \leq \|\delta^\orth\|^2 + \|\delta^0\|^2\,.
	\end{equation}
	To show that $T_\xi$ is nonexpansive it is therefore sufficient to show that for all $\delta,p$ that satisfy \eqref{eqn:fnecondsplit} it holds that $\eqref{eqn:TdiffnormSplit} \leq \|\delta^\orth\|^2 + \|\delta^0\|^2$. Notice, that \eqref{eqn:TdiffnormSplit} and \eqref{eqn:fnecondsplit} do not contain any coupling terms between range- and nullspace portions of $\delta$ or $p$, therefore we can argue about them separately.
	\begin{enumerate}[label=\roman*),wide]
		\item\label{niceopProofCase1} We consider all $\delta^\orth, p^\orth$ such that $0 \leq \|p^\orth\|^2 \leq \langle p^\orth,\delta^\orth \rangle \leq \|\delta^\orth\|^2$. Using \eqref{eqn:Meig} implies that $\|\Lambda^{-1}\|_2 \leq 1$ thus we obtain $\|p^\orth_\Lambda\| \leq \|p^\orth\| \leq \|\delta^\orth\|$ and $\langle \delta^\orth,p^\orth_\Lambda\rangle \leq \|\delta^\orth\|^2$. Therefore, $\tfrac{1}{4}\|\delta^\orth\|^2 + \tfrac{1}{2}\langle \delta^\orth,p^\orth_\Lambda\rangle + \tfrac{1}{4}\|p^\orth_\Lambda\|^2 \leq \tfrac{1}{4}\|\delta^\orth\|^2 + \tfrac{1}{2}\|\delta^\orth\|^2 + \tfrac{1}{4}\|\delta^\orth\|^2=\|\delta^\orth\|^2$, i.e., the result holds for the range space portion.
		\item\label{niceopProofCase2} We consider all $\delta^0, p^0$ such that $0 \leq \|p^0\|^2 \leq \langle p^0,\delta^0 \rangle \leq \|\delta^0\|^2$. It follows that $\|\delta^0\|^2 - 2\langle \delta^0,p^0\rangle + \|p^0\|^2 \leq	 \|\delta^0\|^2 - 2\langle \delta^0,p^0\rangle + \langle \delta^0,p^0\rangle \leq \|\delta^0\|^2$. Thus, the result also holds for the null space portion.
	\end{enumerate}
	Cases \ref{niceopProofCase1} and \ref{niceopProofCase2} together imply the result.
\end{proof}

\section{Condition for satisfying \refass{ass:Z}} \label{app:assCheck}
We give an algorithm to check whether Assumption~\refass{ass:Z} is satisfied at every point in $\set{Z}$, representing a necessary and sufficient condition. This conditions relies on the enumeration of all possible combinations of active components and their active sets, and can be checked \emph{once and offline} for all parameters $\theta$.
We consider $\set{E} := \{ z \in \reals{n} \sep{} Az = b\}$, with $A \in \reals{m \times n}$, and $\set{Z}(\theta)$ with components $\set{Z}^i(\theta)$ defined as
\begin{equation} \label{eqn:assCheckZ}
	\set{Z}(\theta) = \bigcup_{i=1}^{I} \set{Z}^i(\theta) = \bigcup_{i=1}^{I} \{z \in \reals{n} \sep{} G^iz=g^i(\theta)\,, F^iz \leq f^i(\theta)\} \,,
\end{equation}
where $G^i \in \reals{p_i\times n}$ and $F^i \in \reals{m_i\times n}$. The functions $f^i(\theta) = F^{i,\theta}\theta+f^i$ and $g^i(\theta) = G^{i,\theta}\theta+g^i$ depend affinely on a parameter $\theta \in \reals{p}$. With $\set{Z}^i(\theta)$ closed convex polyhedra for fixed $\theta \in \reals{p}$ and non-empty for some $\theta$.

Given a parameter $\theta$ and index $i \in \{1,\ldots,I\}$. The \emph{active set} $\activeSet{\set{Z}^i(\theta)}(z) := \big\{ j \in \{1,\ldots,m_i\} \sep{\big} \suchthat F^i_{j}z = f^i_j(\theta) \big\}$ of $\set{Z}^i(\theta)$ at $z \in \set{Z}^i(\theta)$ is defined in the usual way \cite[Definition~12.1, p.~308]{numericalOptimization2006}, where $F^i_j$ denotes the $j$-th row of $F^i$. For $z \not\in \set{Z}^i(\theta)$, we define $\activeSet{\set{Z}^i(\theta)}(z) := \varnothing$. For any active set $\set{J} \subseteq \{1,\ldots,m_i\}$, $F^i_{\activeSet{}}$ denotes the subset of rows of $F^i$ that belong to the active set $\activeSet{}$, $f^i_{\activeSet{}}$ is defined in the same way. We further define the active set $\activeSet{\set{Z}(\theta)}(z)$ of $\set{Z}(\theta)$ at $z \in \set{Z}(\theta)$, as the collection of all the active sets of the components of $\set{Z}(\theta)$, i.e., $\activeSet{\set{Z}(\theta)}(z) := \{ \activeSet{\set{Z}^i(\theta)}(z) \}_{i=1}^I$,

The set $\activePoly{\set{Z}} \subseteq 2^{\{1,\ldots,I\}}$ denotes the set of sets of active components of $\set{Z}(\theta)$, i.e., $\mathcal{I} \in \activePoly{\set{Z}}$ if and only if there exists $\theta  \in \reals{p}$ and $z \in \set{Z}(\theta)$ such that $\mathcal{I} \equiv \activePoly{\set{Z}(\theta)}(z)$, where $\activePoly{\set{Z}(\theta)}(z)$ denotes the set of active components at $z$, as defined in \refsec{sec:locopt}. Similarly we denote by $\activeSet{\set{Z}}  := \left\{ \activeSet{\set{Z}^i} \right\}_{i=1}^I$ the collection of the sets of all possible active sets of $\set{Z}(\theta)$, where $\activeSet{\set{Z}^i} := \big\{ \activeSet{} \subseteq \{1,\ldots,m_i\} \sep{\big} \exists \theta \in \reals{p}\,, \exists z\in \set{Z}(\theta) \suchthat \activeSet{} \equiv \activeSet{\set{Z}^i(\theta)}(z) \big\}$ denotes the set of all possible active sets for component $i$.

Assumption~\refass{ass:Z} is purely geometric and needs to be satisfied at every point in $\set{Z}$. \refalg{alg:assZCheck} will effectively enumerate all possible combinations of active components and their active sets and check for each such combination. This amounts to a finite number of checks. Checking a parametric set $\set{Z}(\theta)$, due to the polyhedral structure, requires checking finitely many different non-parametric sets.

\begin{algorithm}	
	\caption{Check of Assumption~\refass{ass:Z}} \label{alg:assZCheck}
  	\begin{algorithmic}[1]
   		\Require{$\set{Z}(\theta)$ as in \eqref{eqn:assCheckZ}}
   		\For{\textbf{each} $\mathcal{I} \in \activePoly{\set{Z}}$, $\{ \activeSet{i} \}_{i=1}^I \in \activeSet{\set{Z}}$}
   			\State $\set{N} \gets \bigcap_{i \in \mathcal{I}} \{v \sep{} v = G^{i\transp} \nu + F^{i\transp}_{\activeSet{i}} \mu \,,\mu \geq 0\}$ \Comment{construct regular normal cone}\label{step:assZCheckN}
   			\State \algorithmicif\ $\set{N} = \{0\}$ \algorithmicthen\ \textbf{continue}
   			\State $\set{R} \gets \bigcup_{i \in \mathcal{I}} \{v \sep{} G^iv=0\,, F^i_{\activeSet{i}}v \leq 0\}$ \Comment{construct union of recession cones}\label{step:assZCheckR}
   			\If{$\set{N}^\orth \subseteq \set{R}$} \textbf{continue}\label{alg:assZCheckIncl}
   			\Else
   				\State $w \gets w \in \mathcal{N}^\orth \setminus \set{R}$
   				\State $(z,\theta) \gets z,\theta \suchthat z \in \bigcap_{i=1}^I \set{A}_i(\theta)$
   				\State \Return $(w,z,\theta)$ \Comment{return counter-example} \EndIf
   		\EndFor
  	\end{algorithmic}
\end{algorithm}
\refalg{alg:assZCheck} is a combinatorial algorithm and may perform badly for anything but small dimensions. However when the set $\set{Z}$ is a Cartesian product of sets in low dimension, the check can be performed separately for each part of the Cartesian product. This significantly reduces the ``curse of dimensionality''.
\begin{lemma} \label{lem:assZCheck}
	Given a (parametric) set $\set{Z}(\theta)$ as in \eqref{eqn:assCheckZ}, then \refalg{alg:assZCheck} terminates in a finite number of steps. Either $\set{Z}(\theta)$ satisfies Assumption~\refass{ass:Z} for all $\theta$, or the algorithm returns a valid counter-example.
\end{lemma}
\begin{proof}[Proof of \reflem{lem:assZCheck}]
	We will show that \refalg{alg:assZCheck} effectively checks Assumption \refass{ass:Z} for every $\theta$ and every point in $\set{Z}(\theta)$. We will argue, that checking \refass{ass:Z} for every $\theta \in \reals{p}$ and every $z\in\set{Z}(\theta)$ is equivalent to checking it for every set of active constraints. Note that $\activeSet{\set{Z}}$ the set of active sets of $\set{Z}(\theta)$ is a finite set with at most $2^{\sum_{i=1}^I m_i}$ elements. 
	%We further note that $\activePoly{\set{Z}}$ and $\activeSet{\set{Z}}$ can be constructed by solving a finite number of linear optimization problems.
	
	In \refalg{alg:assZCheck}, we iterate over the elements of $\activePoly{\set{Z}}$ and $\activeSet{\set{Z}}$, which means that the loop is executed at most a finite number of times. Given a pair $\set{I}, \{ \activeSet{i} \}_{i=1}^I$ we want to verify, or invalidate, \refass{ass:Z} by considering two sets $\set{N}$ and $\set{R}$. The sets $\set{N}$, $\set{R}$ are defined solely via the active constraints and do not depend on $\theta$. For each active component $i\in \set{I}$ we define two auxiliary sets
	\begin{align*}
		\set{A}_i(\theta) &:= \{z\sep{} G^iz=g^i(\theta)\,, F_{\activeSet{i}}^iz = f^i_{\activeSet{i}}(\theta)\,,F_{\activeSet{i}^c}^iz < f^i_{\activeSet{i}^c}(\theta) \}\,,\\
		\set{F}_i(\theta) &:= \{z \sep{} G^iz=g^i(\theta)\,, F_{\activeSet{i}}^i z \leq f_{\activeSet{i}}^i(\theta)\}\,,
	\end{align*}
	where the set $\activeSet{i}^c$ is simply the complement of $\activeSet{i}$, i.e., the set of constraints that are not active. For a given $\theta$, $\set{A}_i(\theta)$ is the set of points $z \in \set{Z}$ that belong to the active set $\activeSet{i}$, whereas $\set{F}_i(\theta) \supseteq \set{A}_i(\theta)$ is a larger set, that is defined by only the active constraints and in general is not included in $\set{Z}$. We further define $\set{A}(\theta) :=\bigcap_{i=1}^I \set{A}_i(\theta)$, the set of points that belong to all of the active sets of the active components. By checking Assumption \refass{ass:Z} for each $\set{A}(\theta)$ we will effectively check it for each $z \in \set{Z}$.
	
	The conditions to be checked involve the regular normal cone $\rnormalcone{\set{Z}(\theta)}(z)$. For each $\theta$, $z \in \set{A}(\theta)$ the regular normal cone $\rnormalcone{\set{Z}(\theta)}(z)$ is given as $\rnormalcone{\set{Z}(\theta)}(z) = \bigcap_{i \in \mathcal{I}} \{G^{i\transp} \nu + F^{i\transp}_{\activeSet{i}} \mu \sep{} \mu \geq 0\}$, using \refprop{prop:normalConeCharact} and the fact that normal cones of polyhedral sets can be represented as linear and conic combinations of their halfspaces, see \cite[Theorem~6.46, p.~231]{rockafellar1998}.
	It is independent of $\theta$ and $z$, and only depends the active constraints and components. This set, $\mathcal{N} := \rnormalcone{\set{Z}(\theta)}(z)$, is computed in step \ref{step:assZCheckN} of \refalg{alg:assZCheck}. A half-space representation of $\set{N}$ can be obtained using Fourier-Motzkin elimination \cite[p.~84]{dantzig1963} for each $i \in \set{I}$, after which the intersections are trivial. We now need to check two conditions. If $\mathcal{N} = \{0\}$, then the first part of \refass{ass:Z} is satisfied for all $z \in \set{A}(\theta)$, and we continue with the next active set. Otherwise, we need to check whether there exists an $\epsilon > 0$ such that for all $w \in \mathcal{N}^\orth\cap \ball_\epsilon$ we have that $z+w \in \set{Z}(\theta)$. In order to check this condition, in step~\ref{step:assZCheckR}, we construct the set $\set{R}$ as the union of the recession cones $\set{R}_i := \rec(\set{F}_i(\theta)) = \{v \sep{} G^iv=0\,, F^i_{\activeSet{i}}v \leq 0\}$ of $\set{F}_i(\theta)$. Finally it remains to show that checking this condition for all $z \in \set{A}(\theta)$, in step \ref{alg:assZCheckIncl}, is equivalent to checking whether the orthogonal complement $\mathcal{N}^\orth := \{v \in \reals{n} \sep{} \langle v, u \rangle = 0\,, \forall u \in \mathcal{N}\}$ is contained in $\set{R}$. If $\mathcal{N}^\orth \subseteq \set{R}$, consider any $z \in \set{A}(\theta)$, an $\epsilon>0$ small enough and any $w \in \mathcal{N}^\orth  \cap \ball_\epsilon$. Clearly there exists an $i \in \mathcal{I}$ such that $w \in \set{R}_i$. Furthermore, by construction of $ \set{A}(\theta)$ and $\set{R}_i$, we have $z+w \in \set{Z}^i \subseteq \set{Z}$. This means \refass{ass:Z} is satisfied for all $z \in \set{A}(\theta)$ and we continue with the next active set.
If $\mathcal{N}^\orth \not\subseteq \set{R}$, then there exists a $w \in \mathcal{N}^\orth \setminus \set{R}$. Because $\mathcal{N}$ and $\set{R}_i$ are cones, this is true for all positive scalings of $w$. In particular, for any $\epsilon>0$ there exists a $t>0$ such that $t w \in \mathcal{N}^\orth \cap \ball_\epsilon$ and $t w \not\in \set{R}$. Analogously to above, it follows that for any $z \in \set{A}(\theta)$, $z+t w \not\in \set{Z}^i$ for all $i\in\set{I}$, and therefore $z+t w \not\in \set{Z}$, which violates \refass{ass:Z}. In this case, the algorithm terminates and returns a violating direction $w$, point $z$ and parameter $\theta$. Because \refalg{alg:assZCheck} enumerates all active sets, it either terminates with a counterexample, or it returns after finitely many steps, having checked all possibilities.
\end{proof}

\section{MacMPEC benchmark problems} \label{app:macmpec}
The \texttt{MacMPEC} library \cite{leyffer2009} is a collection of almost 200 benchmark problems for mathematical programs with equilibrium constraints (MPECs). It contains 41 MPCCs with linear constraints (including linear complementarity constraints) that can be formulated in the form of \refprob{prob:ConsensusProblem} with a positive definite quadratic objective.
Formulating such MPCCs in the form of \refprob{prob:ConsensusProblem} can lead to a number of regions that is worst-case exponential in the number of complementarity constraints. Therefore, for 12 of the 41 problems we were unable to obtain a tractable representation in the form of \refprob{prob:ConsensusProblem}. The remaining 29 problems were brought into the form of \refprob{prob:ConsensusProblem}. However, in only two of the cases, \texttt{qpec1} and \texttt{qpec2}, we managed to extract the Cartesian product structure, i.e. $N > 1$, and in all cases, we used $\set{E} = \reals{n}$.
We solved each problem using the proposed method with $\xi = 2\lambda_{\rm max}(H)$, $\epsilon_{\rm tol} = 10^{-6}$ and initial iterate $s_0 = 0$. For all of these problems, \refalg{alg:krasnoselskij} converged and returned solutions close to the global minimum.

In \reftab{tab:macmpec}, for each of the 41 problems, we report the problem dimensions: the number of decision variables $n$, the number of regions $m$ and the ``horizon'' $N$. We also report the objective value of the found local minima compared to the objective value of the global optimum, and how many iterations were needed for each problem to converge to a solution with $10^{-6}$ consensus tolerance. Each of the 29 problems that could be tractably represented converged close to the global minimum within a few iterations, which the exception of \texttt{scale3}, which needed \num{6815} iterations to converge. For each problem, we have additionally, checked Assumption~\refass{ass:Z} using \refalg{alg:assZCheck}. For the problems \texttt{portfl-i-\{1,2,3,4,6\}}, checking Assumption~\refass{ass:Z} was intractable due to the large number of regions (more than \num{4000}) and decision variables (74). For 23 of the remaining 24 problems Assumption~\refass{ass:Z} was determined to be satisfied, with \refalg{alg:assZCheck} terminating within a few seconds. For one problem, \texttt{hs044-i}, \refalg{alg:assZCheck} terminated after $\unit[40.1]{minutes}$ with a counter-example, certifying that Assumption~\refass{ass:Z} does not hold. Nevertheless, \refalg{alg:krasnoselskij} converged for each of the 29 problems, including for the problems \texttt{hs044-i} and \texttt{portfl-i-\{1,2,3,4,6\}}, where Assumption~\refass{ass:Z} was violated or could not be checked. Moreover, they converged to local minima, as indicated by \refthm{thm:localOptimality}. 

%\begin{sidewaystable}
	{\centering
	\input{macmpec.tex}}
%\end{sidewaystable}
\clearpage

\bibliographystyle{siamplain}
\bibliography{references}
\end{document}

%% file: ex_shared.tex
\usepackage{lipsum}
\usepackage[utf8]{inputenc}
\usepackage{shellesc} % This avoids some tikzexternalize problems, see https://tug.org/pipermail/texhax/2016-May/022310.html
\usepackage{cite}
\usepackage[noend]{algpseudocode}
\usepackage{algorithm}
\usepackage{algorithmicx}
\usepackage{booktabs}
\usepackage{mathtools}
\usepackage{multirow}
\usepackage{mdframed}
\usepackage[inline]{enumitem}
\usepackage[normalem]{ulem}
\usepackage{xargs}
\usepackage{subcaption}
\usepackage{pgfplots}
\usepackage{wrapfig}
\usepackage{tikz}
\usepackage{tikzscale}
\usepackage[]{units}
\usepackage{siunitx}
\usetikzlibrary{calc,decorations.pathreplacing,decorations.text,positioning,patterns,topaths,external,3d,spy}
\usepgfplotslibrary{groupplots}
\pgfplotsset{compat=1.9}

\usepackage{ifamath} % Include commonly used math commands not defined by LaTeX

\usepackage{rotating}
\usepackage{longtable} % for 'longtable' environment
\usepackage{lscape} % for 'landscape' environment

\newcommand{\horizon}{N}
\newcommand{\NPwa}{m}

\newcommand{\cost}[1]{\tfrac{1}{2}{#1}^\transp H #1 + h^\transp #1}
\newcommand{\activePoly}[1]{\mathcal{I}_{#1}}
\newcommand{\activeSet}[1]{\mathcal{J}_{#1}}
\newcommand{\fixc}[2]{{\fix #1 | #2}}

\newcommand{\asslabel}[2]{#2\def\@currentlabel{#2}\label{#1}}
\renewcommand{\refass}[1]{(\ref{#1})}

\definecolor{cred}{rgb}{1,0,0}

\definecolor{gurobicolor}{rgb}{0.25,0.25,0.75}
\definecolor{mosekcolor}{rgb}{0.4,0.4,0.75}
\definecolor{cplexcolor}{rgb}{0.55,0.55,0.75}
\definecolor{gurobicolorlight}{rgb}{0.602,0.602,0.867}
\definecolor{mosekcolorlight}{rgb}{0.731,0.731,0.887}
\definecolor{cplexcolorlight}{rgb}{0.856,0.856,0.922}
\definecolor{darkgreen}{rgb}{0.2,0.5,0.2}
\definecolor{lightdarkgreen}{rgb}{0.5,0.8,0.5}
\definecolor{verylightdarkgreen}{rgb}{0.8,0.92,0.8}
\definecolor{color1}{rgb}{0.25,0.25,0.75}
\colorlet{color1light}{color1!55}
\colorlet{color1verylight}{color1!30}
\colorlet{color2}{red!90!black}
\colorlet{color2light}{color2!65}
\colorlet{color2verylight}{color2!20}
\colorlet{color3}{orange}
\colorlet{color3dark}{color3!80!black}
\colorlet{color3light}{color3!55}
\colorlet{color3verylight}{color3!30}

\colorlet{LinkColor}{red}	
\ifpdf
  \DeclareGraphicsExtensions{.eps,.pdf,.png,.jpg}
\else
  \DeclareGraphicsExtensions{.eps}
\fi

\numberwithin{theorem}{section}

\newcommand{\TheTitle}{Low-complexity method for hybrid MPC with local guarantees} 
\newcommand{\TheAuthors}{D. Frick, A. Georghiou, J. L. Jerez, A. Domahidi and M. Morari}

\headers{\TheTitle}{\TheAuthors}

\title{{\TheTitle}\thanks{This manuscript is the preprint of a paper submitted to the SIAM Journal on Control and Optimization. If accepted, the copy of record will be available at http://epubs.siam.org/journal/sjcodc}}

\author{
  Damian Frick\thanks{Automatic Control Laboratory, ETH Zurich, Physikstrasse 3, 8092 Z\"urich, Switzerland
    (\email{dafrick@control.ee.ethz.ch}, \email{morari@control.ee.ethz.ch}).}
  \and
  Angelos Georghiou\thanks{Desautels Faculty of Management, McGill University, Montreal, Quebec, Canada (\email{angelos.georghiou@mcgill.ca}).}
  \and
  Juan L. Jerez\thanks{embotech AG, Technoparkstrasse 1, 8005 Z\"urich, Switzerland (\email{jerez@embotech.com}, \email{domahidi@embotech.com}).}
  \and
  Alexander Domahidi\footnotemark[4]
  \and
  Manfred Morari\footnotemark[2]
}

%% file: racecars_tableDistComp_vielma3.tex
\begin{tabular}{lcccccccccc}
	$N$ & \shortstack{our\\method} & ADMM & Gurobi & CPLEX & MOSEK\\
	\cmidrule[1.5pt](l){1-1} \cmidrule(l){2-2} \cmidrule(l){3-3} \cmidrule(l){4-4}  \cmidrule(l){5-5} \cmidrule(l){6-6}
	10 & 0.8\% (\unit[34]{ms}) & 0.2\% (\unit[21]{ms}) & 3.3\% (\unit[1.1]{s}) & 11\% (\unit[1.8]{s}) & 74\% (\unit[3]{s}) \\
	20 & 0.7\% (\unit[87]{ms}) & 0.3\% (\unit[55]{ms}) & 2.5\% (\unit[10]{s}) & 7.7\% (\unit[12]{s}) & 48\% (\unit[43]{s}) \\
	\bottomrule
\end{tabular}

%% file: macmpec.tex
%\small\addtolength{\tabcolsep}{-5pt}
{%
\begin{landscape}
\begin{longtable}{lrrrrlrrl}
	\toprule
	&  \multicolumn{3}{c}{problem dimensions} & \multicolumn{2}{c}{objective value} & & & \\
	\cmidrule(r){2-4} \cmidrule(r){5-6}
	& $n$ & $m$ & $N$ & our method & optimal & $\xi$ & \#iterations & Assumption~(A3) \\
	\cmidrule(r){2-2} \cmidrule(r){3-3} \cmidrule(r){4-4} \cmidrule(r){5-5} \cmidrule(r){6-6} \cmidrule(r){7-7} \cmidrule(r){8-8} \cmidrule(r){9-9} 
	\endfirsthead
	& \multicolumn{3}{c}{problem dimensions} & \multicolumn{2}{c}{objective value} & & & \\
	\cmidrule(r){2-4} \cmidrule(r){5-6}
	& $n$ & $m$ & $N$ & our method & optimal & $\xi$ & \#iterations & Assumption~(A3) \\
	\cmidrule(r){2-2} \cmidrule(r){3-3} \cmidrule(r){4-4} \cmidrule(r){5-5} \cmidrule(r){6-6} \cmidrule(r){7-7} \cmidrule(r){8-8} \cmidrule(r){9-9} 
	\endhead
	\bottomrule
	\multicolumn{9}{r@{}}{continued \ldots}\\
	\endfoot
	\bottomrule
	\caption{Results for the \texttt{MacMPEC} collection of benchmark problems \cite{leyffer2009}.\\\small
$^{\mathsection}$ the problems \texttt{bard1}, \texttt{ex.9.2.1} and \texttt{ex.9.2.7} are equivalent when formulated in the form of Problem~(1).\\
$^{\dagger}$ the problems \texttt{liswet1-\{050,100,200\}}, \texttt{qpec-100-\{1,2,3,4\}}, \texttt{qpec-200-\{1,2,3,4\}} and \texttt{ralphmod} cover a total 12 problem instances and were excluded because we were unable to find a tractable representation in the form of Problem~(1) due to the large number of complementarity constraints and slack variables.\\
$^{\mathparagraph}$ for \texttt{hs044-i}, Algorithm~(3) returned a counter-example to Assumption~(A3) with $w = (1,2,-1,1)$ at $z = (1.346, 1.336, 3.0707, 2.0587)$, still the proposed method converged to a near optimal solution.\\
$^{\ddagger}$ for the problems \texttt{portfl-i-\{1,2,3,4,6\}}, checking Assumption~(A3) with Algorithm~(3) was intractable due to the large number of regions and decision variables. Nevertheless, the proposed method converged to near optimal solutions.\label{tab:macmpec}
}
	\endlastfoot
	\texttt{bard1}$^{\mathsection}$ & $2$ & $3$ & $1$ & $17$ & $17$ & $16$ & $52$ & satisfied $(<\unit[1]{s})$\\
	\texttt{bard1m} & $2$ & $3$ & $1$ & $17$ & $17$ & $16$ & $52$ & satisfied $(\unit[1.56]{s})$\\
	\texttt{ex.9.2.1}$^{\mathsection}$ & $2$ & $3$ & $1$ & $17$ & $17$ & $16$ & $52$ & satisfied $(<\unit[1]{s})$\\
	\texttt{ex.9.2.2} & $2$ & $1$ & $1$ & $100$ & $100$ & $4$ & $60$ & satisfied $(<\unit[1]{s})$\\
	\texttt{ex.9.2.4} & $2$ & $2$ & $1$ & $0.5$ & $0.5$ & $2$ & $111$ & satisfied $(<\unit[1]{s})$\\
	\texttt{ex.9.2.5} & $2$ & $3$ & $1$ & $5$ & $5$ & $4$ & $113$ & satisfied $(<\unit[1]{s})$\\
	\texttt{ex.9.2.6} & $4$ & $9$ & $1$ & $-1$ & $-1$ & $4$ & $54$ & satisfied $(\unit[46.2]{s})$\\
	\texttt{ex.9.2.7}$^{\mathsection}$ & $2$ & $3$ & $1$ & $17$ & $17$ & $16$ & $52$ & satisfied $(<\unit[1]{s})$\\
	\texttt{gauvin} & $2$ & $2$ & $1$ & $20$ & $20$ & $4$ & $116$ & satisfied $(<\unit[1]{s})$\\
	\texttt{hs044-i} & $4$ & $9$ & $1$ & $15.6178$ & $15.6178$ & $4$ & $111$ & violated$^{\mathparagraph}$ $(\unit[40.1]{min})$\\
	\texttt{jr1} & $2$ & $2$ & $1$ & $0.5$ & $0.5$ & $4$ & $102$ & satisfied $(<\unit[1]{s})$\\
	\texttt{jr2} & $2$ & $2$ & $1$ & $0.5$ & $0.5$ & $4$ & $102$ & satisfied $(<\unit[1]{s})$\\
	\texttt{kth2} & $2$ & $1$ & $1$ & $8.6631e-13$ & $0$ & $4$ & $105$ & satisfied $(<\unit[1]{s})$\\
	\texttt{kth3} & $2$ & $2$ & $1$ & $0.5$ & $0.5$ & $2$ & $105$ & satisfied $(<\unit[1]{s})$\\
	\texttt{liswet1-\{050,100,200\}}$^{\dagger}$ & -- & -- & -- & -- & -- & -- & -- & --\\
	\texttt{portfl-i-1} & $74$ & $4095$ & $1$ & $1.5026e-05$ & $1.502e-05$ & $4$ & $98$ & --$^{\ddagger}$\\
	\texttt{portfl-i-2} & $74$ & $4087$ & $1$ & $1.4583e-05$ & $1.457e-05$ & $4$ & $96$ & --$^{\ddagger}$\\
	\texttt{portfl-i-3} & $74$ & $4093$ & $1$ & $6.2659e-06$ & $6.265e-06$ & $4$ & $97$ & --$^{\ddagger}$\\
	\texttt{portfl-i-4} & $74$ & $4095$ & $1$ & $2.1783e-06$ & $2.177e-06$ & $4$ & $96$ & --$^{\ddagger}$\\
	\texttt{portfl-i-6} & $74$ & $4095$ & $1$ & $2.3625e-06$ & $2.361e-06$ & $4$ & $95$ & --$^{\ddagger}$\\
	\texttt{qpec-100-\{1,2,3,4\}}$^{\dagger}$ & -- & -- & -- & -- & -- & -- & -- & --\\
	\texttt{qpec-200-\{1,2,3,4\}}$^{\dagger}$ & -- & -- & -- & -- & -- & -- & -- & --\\
	\texttt{qpec1} & $30$ & $2$ & $10$ & $80$ & $80$ & $4$ & $114$ & satisfied $(<\unit[1]{s})$\\
	\texttt{qpec2} & $30$ & $2$ & $10$ & $45$ & $45$ & $4$ & $119$ & satisfied $(<\unit[1]{s})$\\
	\texttt{ralphmod}$^{\dagger}$ & -- & -- & -- & -- & -- & -- & -- & --\\
	\texttt{scholtes3} & $2$ & $2$ & $1$ & $0.5$ & $0.5$ & $2$ & $105$ & satisfied $(<\unit[1]{s})$\\
	\texttt{scholtes5} & $3$ & $3$ & $1$ & $1$ & $1$ & $4$ & $111$ & satisfied $(\unit[1.37]{s})$\\
	\texttt{scale1} & $2$ & $2$ & $1$ & $1$ & $1$ & $40000$ & $70$ & satisfied $(<\unit[1]{s})$\\
	\texttt{scale2} & $2$ & $2$ & $1$ & $1$ & $1$ & $400$ & $105$ & satisfied $(<\unit[1]{s})$\\
	\texttt{scale3} & $2$ & $2$ & $1$ & $1$ & $1$ & $40000$ & $6815$ & satisfied $(<\unit[1]{s})$\\
	\texttt{scale4} & $2$ & $2$ & $1$ & $1$ & $1$ & $40000$ & $70$ & satisfied $(<\unit[1]{s})$\\
	\texttt{scale5} & $2$ & $2$ & $1$ & $100$ & $100$ & $400$ & $105$ & satisfied $(<\unit[1]{s})$\\
	\texttt{sl1} & $2$ & $6$ & $1$ & $0.0001$ & $0.0001$ & $4$ & $54$ & satisfied $(\unit[20.8]{s})$\\
\end{longtable}
\end{landscape}
}